\documentclass[10pt,reqno]{amsart}
\usepackage[utf8]{inputenc}
\usepackage{amsmath,times,hyperref, multicol, mathrsfs}
\usepackage{amsthm}
\usepackage{amssymb}
\usepackage{amsfonts}
\usepackage{url}
\usepackage{latexsym}
\usepackage{graphicx}
\usepackage[usenames,dvipsnames]{color} 
\usepackage{enumitem}
\setlist[enumerate]{leftmargin=*}
\usepackage[font={small}]{caption}
\usepackage{subcaption}
\usepackage{cite}
\usepackage{comment}
\usepackage{mathtools}
\mathtoolsset{showonlyrefs}

\usepackage{xcolor}

\usepackage{tcolorbox}
\tcbuselibrary{theorems}

\newtcbtheorem[number within=section]{mytheo}{Problem}%
{colback=blue!5,colframe=blue!35!black,fonttitle=\bfseries}{th}

\newcommand{\0}{{\color{lightgray}0}}
\newcommand{\A}{\mathrm{A}}
\renewcommand{\S}{\mathrm{S}}

\newcommand{\D}{\mathbb{D}}
\newcommand{\R}{\mathbb{R}}
\newcommand{\Q}{\mathbb{Q}}
\newcommand{\B}{\mathcal{B}}
\newcommand{\C}{\mathbb{C}}
\newcommand{\CC}{\mathscr{C}}

\renewcommand{\leq}{\leqslant}
\renewcommand{\geq}{\geqslant}

\newcommand{\N}{\mathbb{N}}

\newcommand{\ran}{\operatorname{ran}}
\newcommand{\diag}{\operatorname{diag}}
\newcommand{\norm}[1]{\| #1 \|}

\newcommand{\h}{\mathcal{H}}

\newcommand{\V}{\mathcal{V}}

\newcommand{\T}{\mathbb{T}}

\newcommand{\minimatrix}[4]{\begin{bmatrix} #1 & #2 \\ #3 & #4 \end{bmatrix}  }

\renewcommand{\hat}{\widehat}

\renewcommand{\vec}[1]{{\bf #1}}

\renewcommand{\phi}{\varphi}

\let\oldenumerate=\enumerate
	\def\enumerate{
	\oldenumerate
	\setlength{\itemsep}{5pt}
	}
\let\olditemize=\itemize
	\def\itemize{
	\olditemize
	\setlength{\itemsep}{5pt}
	}

\theoremstyle{definition}
\newcounter{problem}
\newtheorem{Theorem}[equation]{Theorem}
\newtheorem{Lemma}[equation]{Lemma}
\newtheorem{Proposition}[equation]{Proposition}

\newtheorem{Definition}[equation]{Definition}
\newtheorem{Problem}[problem]{Problem}
\newtheorem{Example}[equation]{Example}

\newtheorem{Remark}[equation]{Remark}

\allowdisplaybreaks

\numberwithin{equation}{section}

\author{Stephan Ramon Garcia}
\address{Department of Mathematics and Statistics, Pomona College, 610 N. College Ave., Claremont, CA 91711, USA}
\email{stephan.garcia@pomona.edu}

\author{Ryan O'Loughlin}
\address{School of Mathematics, University of Leeds, Leeds, LS2 9JT U.K.}
\email{R.OLoughlin@leeds.ac.uk}

\author{Jiahui Yu}
\address{Department of Mathematics, Massachusetts Institute of Technology, Simons Building, 77 Massachusetts Avenue, Cambridge, MA 02139-4307, USA}
\email{jiahu878@mit.edu}

\begin{document}

\title[Symmetric and Antisymmetric Tensor Products]{Symmetric and Antisymmetric Tensor Products for the Function-Theoretic Operator Theorist}

\thanks{Garcia is supported by NSF Grant DMS-2054002. O'Loughlin is grateful to EPSRC for financial support.}

\subjclass{46B28, 30H10, 47A30, 47A10}
\keywords{Tensor product, symmetric tensor product, antisymmetric tensor product, Hilbert space, Hardy spaces, norm, spectrum, shift operator, diagonal operator.}

\maketitle
\begin{abstract}
We study symmetric and antisymmetric tensor products of Hilbert-space operators, focusing on norms and spectra for some well-known classes favored by function-theoretic operator theorists.  We pose many open questions that should interest the field.
\end{abstract}

\section{Introduction}
Tensor products and their symmetrization have appeared in the literature since the mid-nineteenth century,  such as in Riemann's foundational work on differential geometry \cite{RMN1, RMN2}. 
Tensors describe many-body quantum systems \cite{naber2021quantum} and symmetric tensors underpin the foundations of general relativity \cite{diffgeom}. In a separate yet overlapping vein, multilinear algebra \cite{multbook} and representation theory \cite{repbook} utilize symmetric tensor product spaces.  

Decomposing a symmetric tensor into a minimal linear combination of tensor powers of the same vector arises in mobile communications, machine learning, factor analysis of $k$-way arrays, biomedical engineering, psychometrics, and chemometrics; see\cite{ref2, ref3, ref4, ref5, ref6} and the references therein. We refer the reader to \cite{refmain} for a study of this decomposition problem. Symmetric tensors also arise in statistics \cite{statreference}.

In quantum mechanics, many-body systems are represented in terms of tensor products of wave functions. In the simplest case, where the systems do not interact, the Hamiltonian of the many-body system corresponds to a symmetric tensor product of operators
\cite[Ch.~4, Sec.~9]{Kostrikinbook}. 
Recently there has been an endeavour within the physics community to study selfadjoint extensions of symmetric tensor products of operators  \cite{ext1, ext2, ext3}. Furthermore, the symmetric part of a quantum geometric tensor can be exploited as a tool to detect quantum phase transitions in $\mathcal{PT}$-symmetric quantum mechanics \cite{quantsym}.

Unfortunately, there is little literature about symmetric tensor products of non-normal operators. 
The purpose of this paper is to introduce the basic ideas to the function-theoretic operator theory community. We study some fundamental operator-theoretic questions in this area, such as finding the norm and spectrum of symmetric tensor products of operators. We work through some examples with familiar operators, such as the unilateral shift, its adjoint, and diagonal operators. Given the ramifications of symmetric tensor products in a broad spectrum of fields, we hope that initiating this study will shed new light on classical problems and lead to new directions of study for function-theoretic operator theorists.

The layout of this paper is as follows. Section \ref{Subsection:TensorSpace} introduces symmetric and antisymmetric tensor power spaces, the domains for the operators in Section \ref{Section:TensorOperators}. In Section \ref{Section:Basic} we collect results on operator-theoretic properties of symmetric tensor products of bounded operators.
Section \ref{Section:Norm} is devoted to the norm of such operators, while
Section \ref{Section:Spectrum} focuses on the spectrum.
We study symmetric tensor products of diagonal operators in Section \ref{Section:Diagonal},
the forward and backward shift operators in Section \ref{Chapter:ShiftAndAdjoint}, and  
the symmetric tensor product of shifts and diagonal operators in Section \ref{Chapter:ShiftDiagonal}.
We conclude in Section \ref{Section:Questions} with a host of open questions that should appeal to researchers in function-theoretic operator theory.

\section{Symmetric and antisymmetric tensor power spaces}\label{Subsection:TensorSpace}

Symmetric and antisymmetric tensor powers are familiar in mathematical physics, 
but less so in function-theoretic operator theory.  We summarize the basics, with abbreviated explanations
or without proof; see \cite[Sec.~I.5]{Bhatia} or \cite[Section 3.8]{SimonReal} for the details.

Let $\h$ be a complex Hilbert space, in which the inner product $\langle \cdot, \cdot \rangle$
is linear in the first argument and conjugate linear in the second.
We assume that $\h$ has a countable orthonormal basis. A superscript $^-$ denotes the closure with respect to the norm of $\h$.

Let $\B(\h)$ denote the space of bounded linear operators on $\h$.
For $\vec{u}_1, \vec{u}_2 \ldots , \vec{u}_n \in \h$, the \emph{simple tensor}
$\vec{u}_1 \otimes \vec{u}_2 \otimes \cdots \otimes \vec{u}_n : \h^n \to \C$ acts as follows:
\begin{equation*}
(\vec{u}_1 \otimes \vec{u}_2 \otimes \cdots \otimes \vec{u}_n)(\vec{v}_1,\vec{v}_2, \ldots, \vec{v}_n)
=\langle \vec{u}_1, \vec{v}_1 \rangle \langle \vec{u}_2, \vec{v}_2 \rangle \cdots \langle \vec{u}_n, \vec{v}_n \rangle.
\end{equation*}
A simple tensor is a conjugate-multilinear function of its arguments.
The map taking an $n$-tuple in $\h^n$ to the corresponding simple tensor is linear in each argument.

Let $\h^{ \hat{\otimes}n}$ denote the $\C$-vector space spanned by the simple tensors.
There is a unique inner product on $\h^{ \hat{\otimes}n} $ such that 
\begin{equation}\label{eq:InnerProductMultiply}
    \langle\vec{u}_1 \otimes \vec{u}_2 \otimes \cdots \otimes  \vec{u}_n, \, \vec{v}_1 \otimes \vec{v}_2 \otimes \cdots \otimes \vec{v}_n \rangle :=\langle \vec{u}_1, \vec{v}_1 \rangle \langle \vec{u}_2, \vec{v}_2 \rangle \cdots\langle \vec{u}_n, \vec{v}_n \rangle
\end{equation}
for all $\vec{u}_1, \vec{u}_2, \ldots, \vec{u}_n ,\vec{v}_1, \vec{v}_2, \ldots, \vec{v}_n\in \h$ \cite[Prop.~3.8.2]{SimonReal}.
Moreover,
\begin{equation*}
\norm{\vec{u}_1 \otimes \vec{u}_2 \otimes \cdots \otimes  \vec{u}_n} = \norm{ \vec{u}_1} \norm{\vec{u}_2} \cdots \norm{ \vec{u}_n}.
\end{equation*}

\begin{Definition}[Tensor powers of Hilbert spaces]
Let $\h^{\otimes 0} := \C$. For $n=1,2,\ldots,$ let $\h^{\otimes n}$ denote 
the completion of $\h^{ \hat{\otimes}n}$ with respect to the inner product \eqref{eq:InnerProductMultiply}. 
\end{Definition}

For $n=2$, we may write $\h\otimes \h$ instead of $\h^{\otimes 2}$.
If $\vec{e}_1, \vec{e}_2,\ldots$ is an orthonormal basis for $\h$, then 
$\vec{e}_{i_1} \otimes \vec{e}_{i_2} \otimes \cdots \otimes \vec{e}_{i_n}$ for $(i_1,i_2,\ldots,i_n) \in \N^n$
is an orthonormal basis for $\h^{\otimes n}$.  Here $\N := \{1,2,3,\ldots\}$ denotes the set of natural numbers.

Let $\Sigma_n$ be the group of permutations of $[n]:=\{1,2, \ldots,n\}$.
For all $\pi \in \Sigma_n$ and $\vec{u}_1,\vec{u}_2, \ldots, \vec{u}_n \in \h$, define
\begin{equation}\label{eq:PiSimple}
\hat{\pi}(\vec{u}_1 \otimes \vec{u}_2 \otimes \cdots \otimes \vec{u}_n )
:=\vec{u}_{\pi (1)} \otimes \vec{u}_{\pi (2)}\otimes \cdots \otimes \vec{u}_{\pi (n)}.
\end{equation}
The density of the span of the simple tensors ensures that $\hat{\pi}$ extends to a bounded linear map on $\h^{\otimes n}$.

\begin{Proposition}\label{Proposition:Unitary}
Let $\pi,\tau \in \Sigma_n$.
(a) $\hat{\pi \tau} = \hat{\pi} \hat{\tau}$.
(b) The map $\hat{\pi}$ on $\h^{\otimes n}$ is unitary.
\end{Proposition}

\begin{proof}
(a) Since the span of the simple tensors is dense in $\h^{\otimes n}$, it suffices to observe that
\begin{align*}
\hat{\pi \tau }(\vec{u}_1 \otimes \vec{u}_2 \otimes \cdots \otimes \vec{u}_n )
&=\vec{u}_{(\pi \tau) (1)} \otimes \vec{u}_{(\pi \tau) (2)}\otimes \cdots \otimes \vec{u}_{(\pi \tau) (n)} \\
&=\hat{\pi} (\vec{u}_{ \tau (1)} \otimes \vec{u}_{\tau (2)}\otimes \cdots \otimes \vec{u}_{ \tau (n)}) \\
&=\hat{\pi} ( \hat{\tau} (\vec{u}_1 \otimes \vec{u}_2 \otimes \cdots \otimes \vec{u}_n ) )
\end{align*}
for any $ \vec{u}_1,\vec{u}_2,\ldots, \vec{u}_{n} \in \h$.

\medskip\noindent(b) For any $ \vec{u}_1,\vec{u}_2,\ldots, \vec{u}_{n},\vec{v}_1, \vec{v}_2 ,\ldots, \vec{v}_n \in \h$,
\eqref{eq:InnerProductMultiply} ensures that
\begin{align*}
& \langle\hat{\pi}(\vec{v}_1\otimes \vec{v}_2\otimes \cdots \otimes \vec{v}_n), \, \vec{u}_1 \otimes \vec{u}_2 \otimes \cdots \otimes \vec{u}_n\rangle \\
&\qquad=\langle\vec{v}_{\pi(1)}\otimes \vec{v}_{\pi(2)}\otimes \cdots \otimes \vec{v}_{\pi(n)}, \, \vec{u}_1 \otimes \vec{u}_2 \otimes \cdots \otimes \vec{u}_n\rangle \\
&\qquad= \prod_{i=1}^{n} \langle \vec{v}_{\pi(i)}, \vec{u}_i \rangle 
= \prod_{j=1}^{n} \langle \vec{v}_{j}, \vec{u}_{\pi^{-1}(j)} \rangle \\
&\qquad=\langle\vec{v}_{1}\otimes \vec{v}_2 \otimes \cdots \otimes \vec{v}_{n}, \, \vec{u}_{\pi^{-1}(1)} \otimes \vec{u}_{\pi^{-1}(2)} \otimes \cdots \otimes \vec{u}_{\pi^{-1}(n)}\rangle \\
&\qquad= \big\langle\vec{v}_1\otimes \vec{v}_2 \otimes \cdots \otimes \vec{v}_n, \, \hat{\pi^{-1}}(\vec{u}_1 \otimes  \vec{u}_2 \otimes \cdots \otimes \vec{u}_n)\big\rangle .
\end{align*}
Therefore, $\hat{\pi}^* = \hat{\pi^{-1}}$ and hence $\hat{\pi}^{-1} = \hat{\pi}^*$ by (a).
\end{proof}

We now define certain subspaces of $\h^{\otimes n}$ that respect
the action of the operators $\hat{\pi}$.

\begin{Definition}[Symmetric and antisymmetric tensor powers of Hilbert spaces]\hfill
\begin{enumerate}
\item Let $\h^{\odot 1} :=\h$ and
$\h^{\odot n} :=\{ \vec{v} \in \h^{\otimes n}: \text{$\pi(\vec{v})= \vec{v}$ for all $\pi \in \Sigma_n$} \}$
for $n \geq 2$.

\item Let $\h^{\wedge 1} := \{ \vec{0}\}$ and
$\h^{\wedge n} :=\{ \vec{v} \in \h^{\wedge n}: \text{$\pi(\vec{v})= -\vec{v}$ for all $\pi \in \Sigma_n$} \}$
for $n \geq 2$.
\end{enumerate}
\end{Definition}

We may write $\h \odot \h$ and $\h \wedge \h$ instead of $\h^{\odot 2}$ and $\h^{\wedge 2}$, respectively.
In this case, there is only one nonidentity $\pi \in \Sigma_2$.

\begin{Example}\label{Example:Hardy}
Let $H^2(\D)$ denote the Hardy space on the unit disk $\D$.
The monomials $1,z,z^2,\ldots$ are an orthonormal basis for $H^2(\D)$,
so the simple tensors $z^i \otimes z^j$ for $i,j=0,1,\ldots$ are an orthonormal basis for
$H^2(\D) \otimes H^2(\D)$.
The unitary map $z^i \otimes z^j \mapsto z^i w^j$ identifies $H^2(\D) \otimes H^2(\D)$ with $H^2(\D^2)$,
the Hardy space on the bidisk $\D^2$ \cite{bidisk1}.
Thus, we identify $H^2(\D) \odot H^2(\D)$ and $H^2(\D) \wedge H^2(\D)$ with
\begin{equation}\label{eq:H2Sym}
H^2_{\operatorname{sym}}(\D^2) := \{  f(z,w) \in H^2(\D^2) : \text{$f(z,w) =  f(w,z)$ for all $z,w \in \D$} \}
\end{equation}
and
\begin{equation}\label{eq:H2Anti}
H^2_{\operatorname{asym}}(\D^2) := \{  f(z,w) \in H^2(\D^2) : \text{$f(z,w) = - f(w,z)$ for all $z,w \in \D$} \},
\end{equation}
respectively.
We freely use these identifications in what follows.  
\end{Example}

\begin{Definition}[Symmetrization and antisymmetrization operators]
Let $\operatorname{sgn}\pi$ denote the sign of a permutation $\pi \in \Sigma_n$.
Define $\A_n: \h^{\otimes n} \to \h^{\otimes n}$ and $\S_n: \h^{\otimes n}\to\h^{\otimes n}$ by 
\begin{equation}\label{eq:SnAn}
\S_{n}:=\frac{1}{n !} \sum_{\pi \in \Sigma_{n}} \hat{\pi}
\qquad \text{and} \qquad
\A_{n}:=\frac{1}{n !} \sum_{\pi \in \Sigma_{n}}\text{sgn}({\pi}) \hat{\pi}.
\end{equation}
\end{Definition}

\begin{Proposition}\label{Proposition:OrthoProj}\hfill
\begin{enumerate}
\item $\S_n$ is the orthogonal projection from $\h^{\otimes n}$ onto $\h^{\odot n}$.
\item $\A_n$ is the orthogonal projection from $\h^{\otimes n}$ onto $\h^{\wedge n}$.
\end{enumerate}
In particular, $\h^{\odot n}$ and $\h^{\wedge n}$ are closed subspaces of $\h^{\otimes n}$.
\end{Proposition}

\begin{proof}
(a) Use Proposition \ref{Proposition:Unitary} and the fact that $\hat{\pi}\S_n = \S_n$ for all $\pi \in \Sigma_n$
to show that $\S_n^2 = \S_n = \S_n^*$ and $\ran \S_n = \h^{\odot n}$.  The proof of (b) is similar.
\end{proof}

Let $\vec{v}_1, \vec{v}_2,\ldots, \vec{v}_n\in \h$ and define the simple \emph{symmetric} and \emph{antisymmetric tensors}
\begin{align*}
\vec{v}_1 \odot \vec{v}_2 \odot \cdots \odot \vec{v}_n
&:=\S_n ( \vec{v}_1 \otimes \vec{v}_2 \otimes \cdots \otimes \vec{v}_n) \quad \text{and}\\
\vec{v}_1 \wedge \vec{v}_2 \wedge \cdots \wedge \vec{v}_n
&:=\A_n ( \vec{v}_1 \otimes \vec{v}_2 \otimes \cdots \otimes \vec{v}_n).
\end{align*}
A factor of $1/\sqrt{n!}$ is included in some sources \cite[(3.8.33)]{SimonReal} and $\vee$ is sometimes used instead of $\odot$.
Note that $\vec{v}_1 \wedge \vec{v}_2 \wedge \cdots \wedge \vec{v}_n = \vec{0}$
if $\vec{v}_i = \vec{v}_j$ for some $i \neq j$.  

\begin{Proposition}\label{Proposition:ONB}
Let $\vec{e}_1, \vec{e}_2, \vec{e}_3,\ldots $ be an orthonormal basis for $\h$. 
\begin{enumerate}
\item $\vec{e}_{i_1} \odot \vec{e}_{i_2} \odot \cdots \odot \vec{e}_{i_n}$ for $1 \leq i_1 \leq i_2 \leq \ldots \leq i_n$
form an orthogonal basis for $\h^{\odot n}$.
\item $\vec{e}_{i_1} \wedge \vec{e}_{i_2} \wedge \cdots \wedge \vec{e}_{i_n}$ for $1 < i_1 < i_2 < \ldots < i_n$
form an orthogonal basis for $\h^{\wedge n}$.
\end{enumerate}
\end{Proposition}

We say ``orthogonal'' instead of ``orthonormal'' because the vectors described in the previous
proposition need not be unit vectors.
Let $m_{\ell}$ denote the number of occurrences of $\ell$ in $(i_1,i_2,\ldots,i_n) \in [d]^n$.  Then there are
${m_1! m_2! \cdots m_d!}$ permutations of $\vec{e}_{i_1} \otimes \vec{e}_{i_2} \otimes \cdots \otimes \vec{e}_{i_n}$ that give rise to the same simple tensor. Thus,
\begin{equation*}
\norm{\vec{e}_{i_1} \odot \vec{e}_{i_2} \odot \cdots \odot \vec{e}_{i_n}}
= \bigg(\frac{m_1! m_2! \cdots m_r!} {n!}\bigg)^{1/2}.
\end{equation*}
If $\dim \h = d$ is finite, then (using the notation for binomial coefficients)
\begin{equation*}
\dim \h^{\odot n} =\binom{d+n-1}{n}
\quad \text{and} \quad
\dim \h^{\wedge n} = 
\begin{cases}
\binom{d}{n} & \text{if $n \leq  d$},\\[2pt]
0 & \text{if $n > d$}.
\end{cases}
\end{equation*}

The case $n=2$ is special since 
$\dim \h^{\otimes 2} = d^2 = \binom{d+1}{2} + \binom{d}{2} = \dim \h^{\odot 2} + \dim \h^{\wedge 2}$,
which suggests Proposition \ref{Proposition:TwoSum}.
The simple symmetric and antisymmetric tensors are 
\begin{equation*}
\vec{u}\odot \vec{v}=\tfrac{1}{2}( \vec{u}\otimes \vec{v} + \vec{v} \otimes \vec{u}) \in \h^{\odot 2} \quad \text{and}\quad  
\vec{u}\odot \vec{v}=\tfrac{1}{2}( \vec{u}\otimes \vec{v} - \vec{v} \otimes \vec{u}) \in \h^{\wedge 2} 
\end{equation*}
for $\vec{u}, \vec{v} \in \h$.
If $\vec{e}_1, \vec{e}_2, \vec{e}_3,\ldots $ is an orthonormal basis for $\h$, then
\begin{enumerate}
    \item $\sqrt{2}(\vec{e}_i \odot \vec{e}_j)$ for $i < j$ and $\vec{e}_i \odot \vec{e}_i$ for $i \geq 1$ 
    form an orthonormal basis for $\h^{\odot 2}$, and

    \item $\sqrt{2}(\vec{e}_i \wedge \vec{e}_j)$  for $i < j$ form an orthonormal basis for $\h^{\wedge 2}$.
\end{enumerate}

\begin{Proposition}\label{Proposition:TwoSum}
$\h^{\otimes 2}= \h^{\odot 2} \oplus \h^{\wedge 2}$ is an orthogonal decomposition.
\end{Proposition}

\begin{proof}
Let $\pi$ be the nonidentity permutation in $\Sigma_2$. 
If $\vec{x} \in \h^{\otimes 2}$, then
\begin{equation}\label{orthogonalproj}
    \vec{x}
    =\underbrace{\tfrac{1}{2}(\vec{x}+\hat{\pi}(\vec{x})) }_{\S_2(\vec{x}) \in \h^{\odot 2}}
    +\underbrace{\tfrac{1}{2}(\vec{x}-\hat{\pi}(\vec{x})) }_{\A_2(\vec{x}) \in \h^{\wedge 2}}
\end{equation}
and hence $\h^{\otimes 2}= \h^{\odot 2} + \h^{\wedge 2}$.
Since $\hat{\pi}$ is unitary (Proposition \ref{Proposition:Unitary}) and
involutive (since $\hat{\pi}^2 =I$), it is selfadjoint.  
Let $\vec{u} \in \h^{\odot 2}$ and $\vec{v} \in \h^{ \wedge 2}$, then 
$
\langle \vec{u}, \vec{v}\rangle =\langle \hat{\pi}\vec{u},\vec{v}\rangle 
=\langle \vec{u},\hat{\pi}\vec{v} \rangle =\langle \vec{u},-\vec{v}\rangle =-\langle \vec{u},\vec{v}\rangle$,
so $\langle \vec{u},\vec{v}\rangle =0$ .  Thus,  $\h^{\odot 2}\cap \h^{\wedge 2} =\{ \vec{0} \}$, and $\h^{\odot 2} \perp \h^{\wedge 2}$.
\end{proof}

\begin{Example}
Recall from Example \ref{Example:Hardy} the identification of $H^2(\D) \otimes H^2(\D)$ with $H^2(\D^2)$.
The orthogonal decomposition of Proposition \ref{Proposition:TwoSum} becomes
\begin{equation}\label{eq:H2DirectSum}
H^2(\D^2) = H^2_{\operatorname{sym}}(\D^2) \oplus H^2_{\operatorname{asym}}(\D^2),
\end{equation}
where the direct summands are defined by \eqref{eq:H2Sym} and \eqref{eq:H2Anti}, respectively.
In this context, $z^i w^i$ and $( z^i w^j + z^j w^i)/\sqrt{2}$ for $0 \leq i < j$ 
form an orthonormal basis for $H^2_{\operatorname{sym}}(\D^2)$ and 
$(z^i w^j - z^j w^i)/\sqrt{2}$ for $i < j$ form an orthonormal basis for $H^2_{\operatorname{asym}}(\D^2)$.
\end{Example}

\begin{Lemma}\label{Lemma:SquareSumConverge}
If $\sum_{i \leq j} | a_{ij}|^2 < \infty$, then $\sum_{i \leq j} a_{ij} \vec{e}_i \odot \vec{e}_j \in \h \odot \h $.
\end{Lemma}

\begin{proof}
Proposition \eqref{Proposition:ONB} ensures that
\begin{equation*}
\Big\|\sum_{i \leq j} a_{ij} \vec{e}_i \odot \vec{e}_j\Big\|^2 
= \Big\|\sum_{i < j} \frac{a_{ij}}{\sqrt{2}} \sqrt{2} \vec{e}_i \odot \vec{e}_j + \sum_{i = 1}^{\infty} a_{ii} \vec{e}_i \odot \vec{e}_i \Big\|^2
\leq  \sum_{i \leq j} |a_{ij}|^2  < \infty. \qedhere
\end{equation*}
\end{proof}

\begin{Lemma}\label{Lemma:VectorBounds}
For $\vec{u}, \vec{v} \in \h$, we have
$\frac{1}{\sqrt{2}}\| \vec{u} \| \| \vec{v} \| \leq \| \vec{u} \odot \vec{v} \| \leq \| \vec{u} \| \| \vec{v} \|$;
both inequalities are sharp.  In particular, the symmetric tensor product of two nonzero vectors is nonzero.
\end{Lemma}

\begin{proof}
The Cauchy--Schwarz inequality provides the upper inequality since
\begin{align}
    \| \vec{u} \odot \vec{v} \|^2  
    &= \tfrac{1}{4} \| \vec{u} \otimes \vec{v} + \vec{v} \otimes \vec{u} \|^2 
    = \tfrac{1}{4} \langle \vec{u} \otimes \vec{v} + \vec{v} \otimes \vec{u}, \, \vec{u} \otimes \vec{v} + \vec{v} \otimes \vec{u} \rangle \\
    &= \tfrac{1}{4} ( \norm{ \vec{u} \otimes \vec{v} }^2 + \norm{\vec{v} \otimes \vec{u}}^2+ |\langle \vec{u} \otimes \vec{v}, \vec{v} \otimes \vec{u}\rangle |^2) \\
    &= \tfrac{1}{4} (2\| \vec{u} \|^2 \| \vec{v} \|^2 + 2 |\langle \vec{u}, \vec{v} \rangle|^2 ) \label{eq:MidEqThing} \\
    & \leq  \tfrac{1}{4} (2\| \vec{u} \|^2 \| \vec{v} \|^2 + 2\| \vec{u} \|^2 \| \vec{v} \|^2) \\
    &= \|\vec{u} \|^2 \|\vec{v} \|^2.
\end{align}
In \eqref{eq:MidEqThing}, $|\langle \vec{u}, \vec{v} \rangle|^2$ is nonnegative, so we obtain the lower inequality.
The upper inequality is sharp if $\vec{u} = \vec{v}$ and the lower inequality is sharp if $\vec{u} \perp \vec{v}$.
\end{proof}

\section{Symmetric and antisymmetric tensor products of operators}\label{Section:TensorOperators}
For $A_1, A_2, \ldots, A_n \in \B(\h)$, define $A_1 \otimes A_2 \otimes \dots \otimes  A_n$ on simple tensors by
\begin{equation*}
    (A_1 \otimes A_2 \otimes \dots \otimes  A_n) (\vec{v}_1 \otimes \vec{v}_2 \otimes \cdots \otimes \vec{v}_n) 
    = A_1 \vec{v}_1 \otimes A_2 \vec{v}_2  \otimes  \cdots \otimes  A_n \vec{v}_n.
\end{equation*}
This extends by linearity to the linear span $\h^{ \hat{\otimes}n}$ of the simple tensors.
The density of $\h^{\hat{\otimes}n}$ in $\h^{\otimes n}$ ensures that
$A_1 \otimes A_2 \otimes \dots \otimes  A_n$ has a unique bounded extension to $\h^{\otimes n}$,
also denoted $A_1 \otimes A_2 \otimes \dots \otimes  A_n$, which satisfies \cite[(3.8.17)]{SimonReal}:
\begin{equation}\label{eq:BoundedExtend}
    \| A_1 \otimes A_2 \otimes \dots \otimes  A_n \| =  \| A_1 \| \| A_2 \| \cdots \|A_n \|.
\end{equation}
We may write $A^{\otimes n}$ instead of $A \otimes A \otimes \cdots \otimes A$ 
($n$ times).

\begin{Proposition}\label{Proposition:Invariant}
Let $A_1, A_2,  \ldots, A_n \in \B(\h)$.  Then $\h^{\odot n}$ and $\h^{\wedge n}$ are invariant under
\begin{equation*}
\S_n(A_1,A_2,\ldots,A_n) = 
    \frac{1}{n!}\sum_{\pi\in \Sigma_n} (A_{\pi(1)} \otimes A_{\pi(2)} \otimes \dots \otimes A_{\pi(n)}) \in \B(\h^{\otimes n}).
\end{equation*}
\end{Proposition}

\begin{proof}
Let $T = \S_n(A_1,A_2,\ldots,A_n)$.
For $\vec{v}_1, \vec{v}_2, \ldots, \vec{v}_n \in \h$,
\begin{align}
&T ( \vec{v}_1 \odot \vec{v}_2 \odot \cdots \odot \vec{v}_n ) 
= \frac{1}{n!}\sum_{\pi\in \Sigma_n}(A_{\pi(1)}  \otimes A_{\pi(2)}  \dots \otimes A_{\pi(n)})  (\vec{v}_1 \odot \vec{v}_2 \odot  \cdots \odot \vec{v}_n) \\
&\qquad= \frac{1}{n!}\sum_{\pi\in \Sigma_n}(A_{\pi(1)}  \otimes A_{\pi(2)} \otimes \dots \otimes A_{\pi(n)}) 
\bigg( \frac{1}{n !} \sum_{\tau \in \Sigma_{n}} \vec{v}_{\tau(1)}  \otimes \vec{v}_{\tau(2)} \otimes \cdots \otimes \vec{v}_{\tau(n)} \bigg) \\[-5pt]
&\qquad= \frac{1}{(n!)^2}\sum_{\pi\in \Sigma_n} \sum_{\tau \in \Sigma_{n}} (A_{\pi(1)}\vec{v}_{\tau(1)}  \otimes A_{\pi(2)}\vec{v}_{\tau(2)} \otimes  \dots \otimes A_{\pi(n)}\vec{v}_{\tau(n)}) \\
&\qquad= \frac{1}{n!}\sum_{\sigma\in \Sigma_n} A_{\sigma(1)}\vec{v}_{1}  \odot A_{\sigma(2)}\vec{v}_{2} \odot \dots \odot A_{\sigma(n)}\vec{v}_{n} ,
\end{align}
a sum of elements in $\h^{\odot n}$.  The density of the simple symmetric tensors in 
$\h^{\odot n}$ ensures that $T \h^{\odot n} \subseteq \h^{\odot n}$. 
The proof that $T \h^{\wedge n} \subseteq \h^{\wedge n}$ is similar. 
\end{proof}

The proposition above suggests the following definition.

\begin{Definition}[Symmetric tensor products of operators]
Let $A_1, A_2, \ldots, A_n \in \B(\h)$.  Then $A_1 \odot A_2 \odot \cdots \odot A_n$
and $A_1 \wedge A_2 \wedge \ldots \wedge A_n$ are the restrictions of 
\begin{equation*}
\S_n(A_1,A_2,\ldots,A_n) = \frac{1}{n!}\sum_{\pi\in \Sigma_n} (A_{\pi(1)} \otimes A_{\pi(2)} \otimes \dots \otimes A_{\pi(n)})
\end{equation*}
to $\h^{\odot n}$ and $\h^{\wedge n}$, respectively.
We may write $A^{\odot n}$ and $A^{\wedge n}$ instead of $A \odot A \odot \cdots \odot A$ 
($n$ times) and $A \wedge A \wedge \cdots \wedge A$ ($n$ times), respectively.
\end{Definition}

Symmetric tensor products are permutation invariant:
\begin{equation}\label{eq:PermInvariance}
A_1 \odot A_2 \odot \cdots \odot  A_n
= A_{\pi(1)} \odot A_{\pi(2)} \odot \cdots \odot  A_{\pi(n)}
\quad \text{for all $\pi \in \Sigma_n$}.
\end{equation}
If $A, B , C \in \B(\h)$, then the domain of $A \odot B$ is $\h \odot \h$, which is not equal to $\h$, the domain of $C$. 
Thus, $( A \odot B ) \odot C$ is not well defined.
Note that $I \odot I \odot \cdots \odot I = I$.

\begin{Proposition}\label{Proposition:BasicNormInequality}
(a) For all $A_1,  A_2,  \ldots , A_n \in \B (\h)$, we have
$\| A_1 \odot A_2 \odot \cdots \odot A_n \| \leq \|A_1\| \|A_2 \| \cdots \| A_n \|$.
(b) For all $A \in \B(\h)$, we have $\| A^{\odot n} \| = \| A \|^n$.
\end{Proposition}

\begin{proof}
(a) Since $ A_1 \odot A_2 \odot \cdots \odot A_n $ is the restriction of 
$ \frac{1}{n!}\sum_{\pi\in \Sigma_n}(A_{\pi(1)}  \otimes A_{\pi(2)} \otimes \dots \otimes A_{\pi(n)}) $ to $\h^{\odot n}$,
its norm is at most
\begin{align}
     \Big\| \frac{1}{n!}\sum_{\pi\in \Sigma_n}(A_{\pi(1)}   \otimes \dots \otimes A_{\pi(n)}) \Big\| 
    &\leq \frac{1}{n!}\sum_{\pi\in \Sigma_n}\|(A_{\pi(1)}  \otimes A_{\pi(2)} \otimes \dots \otimes A_{\pi(n)}) \| \\
    &= \frac{1}{n!}\sum_{\pi\in \Sigma_n}\|A_{\pi(1)}  \| \| A_{\pi(2)} \|  \cdots \| A_{\pi(n)} \| \\ 
    &= \frac{1}{n!}\sum_{\pi\in \Sigma_n}\|A_{1}  \| \| A_{2} \|  \dots \| A_{n} \| \\
    &= \|A_1\| \|A_2 \| \ldots \| A_n \|.
\end{align}

\noindent(b)  First we have $\| A^{\odot n} \| \leq \| A \|^n$ from (a).  Then
\begin{equation*}
    \|A \|^n 
    = \sup_{ \substack{ \vec{v} \in \h \\ \| \vec{v}\| = 1}} \|A\vec{v} \|^n 
    = \sup_{ \substack{ \vec{v} \in \h \\ \| \vec{v}\| = 1}} \| A^{\otimes n} (\vec{v} \otimes \vec{v} \otimes \cdots \otimes \vec{v}) \| 
    \leq \| A^{\odot n} \|. \qedhere
\end{equation*}
\end{proof}

\begin{Example}
If $A,B \in \B(\h)$, then Propositions  \ref{Proposition:TwoSum} and \ref{Proposition:Invariant} ensure that
\begin{equation}\label{eq:UseLots}
    \frac{1}{2}( A \otimes B + B \otimes A) 
    = 
\begin{bmatrix}
A \odot B & \0 \\
\0 & A \wedge B \\
\end{bmatrix} : 
\begin{bmatrix}
\h \odot \h \\
\h \wedge \h
\end{bmatrix} \to 
\begin{bmatrix}
\h \odot \h \\
\h \wedge \h
\end{bmatrix}.
\end{equation}
\end{Example}

\begin{Example}
Let $\h = H^2( \D)$ and let $T_g: H^2(\D) \to H^2( \D)$ be the Toeplitz operator with symbol $g \in L^{\infty}(\D)$. 
Then \eqref{eq:BoundedExtend} ensures that $T_{g} \otimes T_{g} : H^2(\D^2) \to H^2(\D^2)$, the linear extension
of the map $z^i w^j \mapsto T_{g}(z^i) T_{g}(w^j)$, has norm $\|g\|_{\infty}^2$. 
Proposition \ref{Proposition:BasicNormInequality} says that $T_{g} \odot T_{g}$, the restriction of $T_{g} \otimes T_{g}$ to $H^2_{\operatorname{sym}}(\T^2)$, also has norm $\|g\|_{\infty}^2$.
\end{Example}

\section{Basic properties}\label{Section:Basic}
In this section we collect some results on the operator-theoretic properties of symmetric tensor products of bounded Hilbert-space operators.

\begin{Lemma}\label{prodsym}
$(A \odot B)(C\odot D)=\frac{1}{2}(AC\odot BD+AD\odot BC)$
for $A,B,C,D \in \B(\mathcal{\h})$.
\end{Lemma}

\begin{proof}
Restrict
$\frac{1}{2} (A\otimes B+B\otimes A ) \frac{1}{2} (C\otimes D+D\otimes C)
= \frac{1}{4}(AC \otimes BD+ BD \otimes AC+ AD \otimes BC+ BC \otimes AD)$
to $\h \odot \h$ and obtain the desired formula.
\end{proof}

\begin{Example}
Equip $\C^2$ with the standard basis $\vec{e}_1,\vec{e}_2$ and consider
\begin{equation}\label{eq:A2x2}
A = \big[\begin{smallmatrix}
a_{11} & a_{12} \\
a_{21} & a_{22} 
\end{smallmatrix}\big] \quad \text{and} \quad B = \big[\begin{smallmatrix}
b_{11} & b_{12} \\
b_{21} & b_{22} 
\end{smallmatrix}\big].
\end{equation}
With respect to the orthonormal basis
$\vec{e}_1 \otimes \vec{e}_1,
\vec{e}_1 \otimes \vec{e}_2,
\vec{e}_2 \otimes \vec{e}_1,
\vec{e}_2 \otimes \vec{e}_2$ of $\h\otimes \h$,
we see that $\frac{1}{2} (A \otimes B+B \otimes A)$ has the matrix representation
\begin{equation}\label{eq:2x2Tensor}
\frac{1}{2}\small
\begin{bmatrix}
2 a_{11} b_{11} & a_{11} b_{12} + b_{11} a_{12} & a_{12} b_{11} + b_{12} a_{11} & 2a_{12} b_{12} \\
a_{11} b_{21} + b_{11} a_{21} & a_{11} b_{22} + b_{11} a_{22} & a_{12} b_{21} + b_{12} a_{21} & a_{12} b_{22} + b_{12} a_{22} \\
a_{21} b_{11} + b_{21} a_{11} & a_{21} b_{12} + b_{21} a_{12} & a_{22} b_{11} + b_{22} a_{11} & a_{22} b_{12} + b_{22} a_{12} \\
2a_{21} b_{21} & a_{21} b_{22} + b_{21} a_{22} & a_{22} b_{21} + b_{22} a_{21} & 2a_{22} b_{22} 
\end{bmatrix} .
\end{equation}
With respect to the orthonormal basis $\vec{e}_1 \odot \vec{e}_1, \sqrt{2}(\vec{e}_1 \odot \vec{e}_2), \vec{e}_2 \odot \vec{e}_2$
of $\h \odot \h$, the symmetric tensor product $A \odot B$ has the matrix representation
\begin{equation}\label{eq:2x2to3x3}\small
    \begin{bmatrix}
a_{11}b_{11} &\frac{a_{11} b_{12} + b_{11}a_{12}}{\sqrt{2}} &  a_{12}b_{12} \\
\frac{a_{11} b_{21} + b_{11}a_{21}}{\sqrt{2}}& \frac{a_{11} b_{22} + b_{11}a_{22} + a_{12} b_{21} + b_{12}a_{21}}{2} & \frac{a_{12} b_{22} + b_{12}a_{22}}{\sqrt{2}} \\
a_{21}b_{21}& \frac{a_{21} b_{22} + b_{21}a_{22}}{\sqrt{2}} & a_{22}b_{22}
\end{bmatrix} .
\end{equation}
\end{Example}

\begin{Proposition}\label{Prop:CommonInvariant}
If $A_1, A_2,  \ldots, A_n \in \mathcal{B} (\h)$ have a common invariant subspace $\mathcal{V} \subseteq \h$, then $\odot^n \mathcal{V}$ is invariant for $A_1 \odot A_2 \odot \cdots \odot A_n$.
\end{Proposition}
\begin{proof}
    This follows from Proposition \ref{Proposition:Invariant}.
\end{proof}

\begin{Proposition}\label{Proposition:Adjoint}
Let $A_1,A_2,\ldots,A_n \in \B(\h)$.  Then
$(A_1 \odot A_2 \odot \cdots \odot A_n)^*  = A_1^* \odot A_2^*\odot \cdots \odot A_n^*$
and
$(A_1 \wedge A_2 \wedge \cdots \wedge A_n)^* = A_1^* \wedge A_2^*\wedge \cdots \wedge A_n^*$.
\end{Proposition}

\begin{proof}
Recall that $\h^{\odot n}$ and $\h^{\wedge n}$ are invariant under
$\S_n(A_1,A_2,\ldots,A_n)$.  Since
$(A_1 \otimes A_2 \otimes \cdots \otimes A_n)^* = A_1^* \otimes A_2^* \otimes \cdots \otimes A_n^*$,
the result follows.
\end{proof}

\begin{Remark}\label{Remark:AAStar}
Observe that $(A \odot A^*)^* = A^* \odot A = A \odot A^*$; that is, $A \odot A^*$ is selfadjoint. For example, for the
 $2 \times 2$ matrix $A$ in \eqref{eq:A2x2}, the formula \eqref{eq:2x2to3x3} gives
\begin{equation}
A \odot A^* = \begin{bmatrix}
|a_{11}|^2 &\frac{a_{11} \overline{a_{21}} + \overline{a_{11}}a_{12}}{\sqrt{2}} &  a_{12}\overline{a_{21}} \\
\frac{a_{11} \overline{a_{12}} + \overline{a_{11}}a_{21}}{\sqrt{2}}& \frac{a_{11} \overline{a_{22}} + \overline{a_{11}}a_{22} + |a_{12}|^2 + |a_{21}|^2}{2} & \frac{a_{12} \overline{a_{22}} + \overline{a_{21}}a_{22}}{\sqrt{2}} \\
a_{21}\overline{a_{12}}& \frac{a_{21} \overline{a_{22}} + \overline{a_{12}}a_{22}}{\sqrt{2}} & |a_{22}|^2
\end{bmatrix} .
\end{equation}
\end{Remark}

\begin{Theorem}\label{Theorem:Closure}
Let $A_1,A_2,\ldots,A_n\in \B(\h)$.
\begin{enumerate}
\item If $A_1,A_2,\ldots,A_n$ are selfadjoint, then $A_1 \odot A_2 \odot \cdots \odot A_n$ is selfadjoint.
\item If $A_1,A_2,\ldots,A_n$ are normal and commute, then $A_1 \odot A_2 \odot \cdots \odot A_n$ is normal.
\item If $U \in \B(\h)$ is unitary, then $U \odot U \odot \cdots \odot U$ is unitary.
\end{enumerate}
\end{Theorem}

\begin{proof}
(a) This follows from Proposition \ref{Proposition:Adjoint}.

\medskip\noindent(b) The Fuglede--Putnam theorem \cite{FPthm} ensures that $A_i A_j^* = A_j^* A_i$ for $1\leq i,j\leq n$.  Proposition \ref{Proposition:Adjoint} and a computation establish the normality of 
$A_1 \odot A_2 \odot \cdots \odot A_n$.

\medskip\noindent(c) By Proposition \ref{Proposition:Adjoint},
$(U \odot U \odot \cdots \odot U)^* (U \odot U \odot \cdots \odot U) = I \odot I \odot \cdots \odot I$ and similarly $(U \odot U \odot \cdots \odot U) (U \odot U \odot \cdots \odot U)^*   = I \odot I \odot \cdots \odot I$.
\end{proof}

\begin{Example}
The normal matrices $A = \scriptsize \minimatrix{1}{i}{i}{1}$ and $B = \scriptsize \minimatrix{1}{-1}{1}{1}$ do not commute, but
$
A \odot B =
\Bigg[
\begin{smallmatrix}
 1 & -\frac{1-i}{\sqrt{2}} & -i \\
 \frac{1+i}{\sqrt{2}} & 1 &
   -\frac{1-i}{\sqrt{2}} \\
 i & \frac{1+i}{\sqrt{2}} & 1 \\
\end{smallmatrix}\Bigg]$
is not normal.  Thus, the commutativity hypothesis is necessary in (b).
\end{Example}

\begin{Example}
If $A,B \in \B(\h)$ are selfadjoint and noncommuting, then $A \odot B$ is selfadjoint, and hence normal.
Thus, the converse of (b) is false.
\end{Example}

\begin{Example}
The matrices $A = \scriptsize \minimatrix{1}{0}{0}{1}$ and $B = \scriptsize \minimatrix{0}{-1}{1}{0}$ are unitary, but
$A\odot B= \Bigg[\begin{smallmatrix}
0 & -1/\sqrt{2} &0\\
1/\sqrt{2} & 0 & -1/\sqrt{2}\\
0 & 1/\sqrt{2} &0
\end{smallmatrix}\Bigg]$ is not.
\end{Example}

\begin{Proposition}\label{compproj}
For orthogonal projections $P, Q$, where $P,Q \neq 0,I$ 
and $PQ = QP = 0$, the map $2P\odot Q$ is an orthogonal projection. Furthermore, $2P\odot Q \neq I \odot I , 0$.
\end{Proposition}
\begin{proof}
Since $P$ and $Q$ are selfadjoint, $2 \S_2 (P,Q)$ is selfadjoint.
Since $P Q = QP = 0$, we have $(2 \S_2 (P,Q))^2 = 2 \S_2 (P,Q)$. Thus, $2 \S_2 (P,Q)$ is an orthogonal projection, so $2 \S_2 (P,Q)|_{\h \odot \h} = 2P\odot Q$ is an orthogonal projection. To show $2P\odot Q \neq I \odot I , 0$ observe if 
$\vec{x},\vec{y} \in \h$ are nonzero, $P\vec{x} = \vec{x}$, and $Q\vec{y} =\vec{y}$, then 
$(2P\odot Q )(\vec{x} \odot \vec{y}) = \vec{x} \odot \vec{y}$ and $(2P\odot Q) (\vec{x} \odot \vec{x}) = 0$.
\end{proof}

One can define tensor powers of bounded conjugate-linear operators.
An analogue of Proposition \ref{Proposition:Invariant} shows that 
$\h^{\odot n}$ is invariant under $\S_n(C_1,C_2,\ldots, C_n)$
for any bounded conjugate-linear operators $C_1, C_2,  \ldots, C_n $. 
We say that $C$ is a \emph{conjugation} on $\h$ if $C$ is conjugate linear,
isometric, and involutive.  We say that $T \in B(\h)$ is \emph{$C$-symmetric} if $T = CT^*C$ \cite{CSO1, CSO2, CSO3}.
If $C$ is a conjugation, let $C^{\odot n}$ 
denote the restriction of $C^{\otimes n}$ to $\h^{\odot n}$.

\begin{Proposition}\label{Csymmetric}
Let $C$ be a conjugation on $\h$ and let $A_1, A_2,  \ldots, A_n \in \B(\h)$ be $C$-symmetric.
(a) $A_1 \otimes A_2 \otimes \cdots \otimes A_n $ is $C^{\otimes n}$-symmetric.
(b) $A_1 \odot A_2 \odot \cdots \odot A_n $ is $C^{\odot n}$-symmetric.
\end{Proposition}

\begin{proof}
(a) Since $(A_1 \otimes A_2 \otimes \cdots \otimes A_n)^* = A_1^* \otimes A_2^* \otimes \cdots \otimes A_n^*$, the result follows.

\medskip\noindent (b) Since $C^{\otimes n} (\h^{\odot n}) \subseteq \h^{\odot n}$, it follows that $C^{\odot n}$ is a well-defined conjugation on $\h^{\odot n}$. The desired result follows from part (a) and Proposition \ref{Proposition:Adjoint}.
\end{proof}

\section{Norms and spectral radius}\label{Section:Norm}
In this section we provide various bounds for the norm of symmetric tensor products of operators, as well as a spectral-radius formula
for symmetric tensor powers.

\begin{Theorem}\label{Theorem:Norm}
Let $A,B \in \B(\h)$.
\begin{enumerate}
\item $\frac{1}{\sqrt{2}}\sup_{ \vec{x} \in \h, \|\vec{x}\| = 1}
 { \|A\vec{x}\| \|B\vec{x}\|}  \leq \| A \odot B \| $, and this is sharp.
\item If $A,B \neq 0$, then $A \odot B \neq 0$.
\item $\rho ( A^{\odot n} ) = \rho(A)^n$, in which $\rho(A) := \sup \{ |\lambda| \in \sigma(A) \}$ is the spectral radius
\end{enumerate}
\end{Theorem}

\begin{proof}
(a) If $ \|\vec{x}\| = 1$, then $\vec{x} \otimes \vec{x} \in \h \odot \h$ has norm one, so
Lemma \ref{Lemma:VectorBounds} ensures that
\begin{equation*}
\frac{ \|A\vec{x}\| \|B\vec{x}\|}{\sqrt{2}} \leq \| A\vec{x} \odot B\vec{x} \| = \left\|\frac{(A \otimes B + B \otimes A) (\vec{x} \otimes \vec{x})}{2} \right\| \leq \| A \odot B \|.
\end{equation*}
Equality is attained for
$A = \big[\begin{smallmatrix}
 1 & 0 \\
 0 & 0 \\
\end{smallmatrix}\big]$, 
$B = \big[\begin{smallmatrix}
 0 & 0 \\
 1 & 0 
\end{smallmatrix}\big]$, and
$\vec{x} = \big[ \begin{smallmatrix} 1 \\ 0 \end{smallmatrix} \big]$.  Indeed,
\eqref{eq:2x2to3x3} ensures that
\begin{equation*}
A \odot B 
= 
\bigg[
\begin{smallmatrix}
0 & 0 & 0 \\
\frac{1}{\sqrt{2}} & 0 & 0 \\
0 & 0 & 0 
\end{smallmatrix}
\bigg],
\quad \text{and hence} \quad
\| A \odot B \| = \frac{1}{\sqrt{2}},
\end{equation*}
while $\vec{x}$ is of unit norm and $\|A\vec{x}\| = \| B\vec{x} \| = 1$, so $\frac{ \|A\vec{x}\| \|B\vec{x}\|}{\sqrt{2}} = \frac{1}{\sqrt{2}} = \| A \odot B \|  $.

\medskip\noindent(b) 
Let $A,B \neq 0$.  If there is a unit vector $\vec{x}$ such that 
$A \vec{x} \neq \vec{0}$ and $B \vec{x} \neq \vec{0}$, then (a) ensures that
$0< \frac{1}{\sqrt{2}} \| A\vec{x} \| \| B \vec{x} \| \leq \| A \odot B\|$.
So suppose that $A\vec{x} = \vec{0}$ or $B\vec{x} = \vec{0}$ for all $\vec{x}\in \h$.
Pick $\vec{u}$ such that $ A \vec{u} \neq \vec{0}$ and $\vec{v} $ such that  $ B \vec{v} \neq \vec{0}$. 
Then $B \vec{u} =\vec{0}$ and $A \vec{v} =\vec{0}$; moreover, $\vec{u}\neq -\vec{v}$.
Let $\vec{x} = \frac{\vec{u} + \vec{v}}{ \| \vec{u} + \vec{v} \| }$, then (a) ensures that
\begin{equation*}
0 <  \frac{1}{\sqrt{2}} \frac{ \|A \vec{u} \| \| B \vec{v} \| }{\| \vec{u} + \vec{v} \|^2 } 
=  \frac{1}{\sqrt{2}} \| A\vec{x} \| \| B\vec{x} \|  \leq  \| A \odot B \|.
\end{equation*}
In both cases, $A \odot B\neq 0 $ since it has positive norm.

\medskip\noindent(c) Since $ (A^{\odot n})^k =  (A^k)^{\odot n}$ for each $k \in \N$, 
Proposition \ref{Proposition:BasicNormInequality} ensures that
$ \| (A^{\odot n})^k \| = \|(A^k)^{\odot n} \| = \|A^k\|^n$. Gelfand's formula \cite[Prop.~3.8, Ch.~5]{Conway} yields
$$\rho ( A^{\odot n}) = \inf_{k \in \mathbb{N}} \| (A^{\odot n})^k \|^{\frac{1}{k}}  = \inf_{k \in \mathbb{N}} \|A^k\|^{\frac{n}{k}} =  \rho(A)^n. \qedhere
$$
\end{proof}

In contrast to symmetric tensor products, the antisymmetric products of nonzero operators may be $0$.
If $P$ is a rank-one orthogonal projection, then $P^{\wedge n} = 0$ for $n\geq 2$.

\begin{Theorem}\label{already}
\begin{enumerate}
    \item If $A_1, A_2,  \ldots, A_n \in \mathcal{B} ( \h) $ and the $A_i$ have orthogonal ranges, then 
    $\norm{A_1 \odot A_2 \odot \cdots \odot A_n} \leq \frac{1}{\sqrt{n!}}\norm{A_1}\norm{A_2} \cdots \norm{A_n}$.
For $n=2$, the inequality is sharp.
\item If  $(\ker B)^\perp \subseteq \ker A$ and $\ran B \subseteq (\ran A)^\perp$, then
$\frac{1}{2}\|A\|\|B\| \leq\|A \odot B\| \leq\frac{1}{\sqrt{2}}\norm{A}\norm{B}$.
The inequalities are sharp.
\end{enumerate}
\end{Theorem}

\begin{proof}
(a) Recall that the set of finite sums $\sum_{i=1}^k \vec{v}^1_i\otimes \vec{v}^2_i \otimes \cdots \otimes \vec{v}^n_i$
of simple tensors are dense in $\h^{\otimes n}$.  Take the supremum over such vectors and observe that
\begin{align}
&\Big\|\sum_{\pi\in \Sigma_n} A_{\pi(1)} \otimes A_{\pi(2)} \otimes \dots \otimes A_{\pi(n)}\Big\|\\
&\qquad= \sup \frac{\norm{(\sum_{\pi\in \Sigma_n} A_{\pi(1)} \otimes A_{\pi(2)} \otimes \dots \otimes A_{\pi(n)})(\sum_{i=1}^k \vec{v}^1_i\otimes \vec{v}^2_i \otimes \cdots \otimes \vec{v}^n_i)}}{\norm{\sum_{i=1}^k \vec{v}^1_i\otimes \vec{v}^2_i \otimes \cdots \otimes \vec{v}^n_i}}  \\
&\qquad=\sup \frac{\norm{(\sum_{\pi\in \Sigma_n} \sum_{i=1}^k  (A_{\pi(1)}\vec{v}_i^1 \otimes A_{\pi(2)}\vec{v}_i^2 \otimes \dots \otimes A_{\pi(n)}\vec{v}_i^n)}}{\norm{\sum_{i=1}^k \vec{v}^1_i\otimes \vec{v}^2_i \otimes \cdots \otimes \vec{v}^n_i}} \\
&\qquad=\sup \frac{(\sum_{\pi\in \Sigma_n} \norm{ \sum_{i=1}^k  A_{\pi(1)}\vec{v}_i^1 \otimes A_{\pi(2)}\vec{v}_i^2 \otimes \dots \otimes A_{\pi(n)}\vec{v}_i^n}^2)^{1/2}}{\norm{\sum_{i=1}^k \vec{v}^1_i\otimes \vec{v}^2_i \otimes \cdots \otimes \vec{v}^n_i}} \\
&\qquad= \sup_{\vec{x} \in  \h^{\otimes n} }  \frac{\left( \sum_{\pi \in \Sigma_n }\norm{(A_{\pi(1)} \otimes A_{\pi(2)} \otimes \cdots \otimes A_{\pi(n)} )\vec{x}}^2\right)^{1/2}}{\norm{ \vec{x}}} \label{finalline} \\
&\qquad \leq \sqrt{n!}\norm{A_1}\norm{A_2}\cdots \norm{A_n}.
\end{align}
The prepenultimate equality above follows because the $A_i$ have orthogonal ranges;
the final inequality is due to \eqref{eq:BoundedExtend}.
For $n=2$, the matrices from the proof of Theorem \ref{Theorem:Norm}.a have orthogonal ranges and demonstrate that the inequality is sharp.

\medskip\noindent(b)
Part (a) ensures that the desired upper inequality holds and is sharp.  It suffices to examine the lower inequality.
For $\vec{u}\in  (\ker A)^\perp$ and $\vec{v}\in \ker A$,
\begin{align}
\norm{(A\otimes B+B\otimes A)(\vec{u}\otimes \vec{v}+\vec{v}\otimes \vec{u})}
&= \norm{A\vec{u}\otimes B\vec{v} + B\vec{v}\otimes A\vec{u}}\\
&= (\norm{A\vec{u}\otimes B\vec{v}}^2 + \norm{B\vec{v}\otimes A\vec{u}}^2)^{1/2} 
\end{align}
since $(A\vec{u} \otimes B\vec{v}) \perp (B\vec{v}\otimes A\vec{u} )$.
Then 
\begin{align*}
\norm{A\odot B}
&\geq \sup_{\substack{\vec{u}\in \ker A^{\perp} \\ \vec{v} \in \ker A }}\frac{\norm{(A\odot B)(\vec{u}\otimes \vec{v}+\vec{v}\otimes \vec{u})}}{\norm{\vec{u}\otimes \vec{v}+\vec{v}\otimes \vec{u}}}\\
&= \frac{1}{2}\sup_{\substack{\vec{u}\in \ker A^{\perp} \\ \vec{v} \in \ker A}}\frac{(\norm{A\vec{u}\otimes B\vec{v}}^2+\norm{B\vec{v}\otimes A\vec{u}}^2)^{1/2}}{\sqrt{2}\norm{\vec{v}}\norm{\vec{u}}}\\
&= \frac{1}{2} \sup_{\substack{\vec{u}\in \ker A^{\perp} \\ \vec{v} \in \ker A }}
\frac{ \|A\vec{u} \| \| B\vec{v} \|}{\norm{\vec{u}}\norm{\vec{v}}}
=\frac{1}{{2}}\norm{A}\norm{B}
\end{align*}
since $\norm{A}=\sup_{\vec{u}\in  (\ker A)^\perp}\frac{\norm{A\vec{u}}}{\norm{\vec{u}}}$ and $(\ker B)^\perp \subseteq \ker A$.

To see that the lower inequality is sharp,
let $A = \big[ \begin{smallmatrix}
1 & 0 \\
0 & 0
\end{smallmatrix}\big]$ and $ B = I - A$, so
$(\ker B)^\perp \subseteq \ker A$ and $\ran B \subseteq (\ran A)^\perp$. 
Then Proposition \ref{compproj} yields $\norm{A \odot B} = \frac{1}{2}$.
\end{proof}

\section{Spectrum}\label{Section:Spectrum}
Here we present results on the spectrum of symmetric products of Hilbert-space operators
(the finite-dimensional case is simpler; see \cite[p.~18]{Bhatia}).
We find a complete description in some special cases.
In what follows, $\sigma(A)$, $\sigma_{\textrm{p}}(A)$, and $\sigma_{\textrm{ap}}(A)$ denote the spectrum, point spectrum
and approximate point spectrum of $A$, respectively \cite[Def.~2.4.5]{OTBE}.
For $X, Y \subseteq \C$, let $X + Y := \{ x+y : x \in X, \, y \in Y \}$ and $XY := \{ xy :  x \in X, \, y \in Y \}$.

\begin{Theorem}[Brown--Pearcy \cite{Brown&Pearcy}]\label{Theorem:BrownPearcy}
$\sigma(A\otimes B)=\sigma(A)\sigma(B)$ for all $A,B \in \B(\h)$.
\end{Theorem}

\begin{Proposition}\label{Proposition:SpectrumUnion}
Let $A,B \in \B(\h)$.  
\begin{enumerate}
\item $\sigma (\tfrac{1}{2}(A \otimes B + B \otimes A) ) = \sigma (A \odot B) \cup \sigma (A \wedge B)$.

\item $\sigma_{\textrm{p}} (\tfrac{1}{2}(A \otimes B + B \otimes A) ) = \sigma_{\textrm{p}} (A \odot B) \cup \sigma_{\textrm{p}} (A \wedge B)$.
\end{enumerate}
\end{Proposition}

\begin{proof}
This follows from the direct-sum decomposition \eqref{eq:UseLots}.
\end{proof}

\begin{Theorem}\label{Theorem:SpectrumSame}
Let $A \in \B(\h)$.  
\begin{enumerate}
\item $\sigma(A\odot I) = \frac{1}{2}(\sigma(A)+\sigma(A))$.
\item $\sigma(A\odot A) =  \sigma(A)\sigma(A)$.
\end{enumerate}
\end{Theorem}

\begin{proof}
(a) First observe that
\begin{align*}
\sigma(A\odot I) 
&\subseteq \sigma\big(\tfrac{1}{2} (A\otimes I+I \otimes A) \big) && (\text{Lemma \ref{Proposition:SpectrumUnion}})\\
&= \tfrac{1}{2}(\sigma(A)+\sigma(A)) && (\text{by \cite[Theorem 2.1]{Schechter}}).
\end{align*}

Let $\lambda,\mu \in \sigma_{\textrm{ap}}(A)$.  There are sequences $\{\vec{u}_i\}_{i=1}^{\infty}$ and $\{\vec{v}_i\}_{i=1}^{\infty}$ of unit vectors such that $\norm{(A-\lambda I)\vec{u}_i}\to 0$ and $\norm{(A-\mu I)\vec{v}_i}\to 0$.  Then
Lemma \ref{Lemma:VectorBounds} ensures that
\begin{align}
&\bigg\|\bigg(A\odot I- \Big(\frac{\lambda}{2}+\frac{\mu}{2}\Big)(I\odot I)\bigg)\bigg(\frac{\vec{u}_i\odot \vec{v}_i}{\|\vec{u}_i\odot \vec{v}_i \|}\bigg)\bigg\|\\
&\quad \leq \sqrt{2}\bigg\|\bigg(A\odot I- \Big(\frac{\lambda}{2}+\frac{\mu}{2}\Big)(I\odot I)\bigg) (\vec{u}_i\odot \vec{v}_i) \bigg\|
\\
&\quad= \frac{1}{2\sqrt{2}}\big\| \big(A\otimes I+I \otimes A-(\lambda+\mu)(I\otimes I)\big)(\vec{u}_i\otimes \vec{v}_i+\vec{v}_i\otimes \vec{u}_i) \big\|\\
&\quad= \frac{1}{2\sqrt{2}} \big\|A\vec{u}_i \otimes \vec{v}_i + \vec{u}_i \otimes A \vec{v}_i - \lambda \vec{u}_i \otimes \vec{v}_i - \vec{u}_i \otimes  \mu\vec{v}_i \\
&\qquad\qquad\qquad+ A \vec{v}_i \otimes \vec{u}_i + \vec{v}_i \otimes A \vec{u}_i -  \vec{v}_i \otimes  \lambda\vec{u}_i - \mu \vec{v}_i \otimes \vec{u}_i  \big\| \\
&\quad= \frac{1}{2\sqrt{2}} \big\| (A-\lambda I)\vec{u}_i \otimes \vec{v}_i + \vec{v}_i \otimes (A-\lambda I)\vec{u}_i \\
&\qquad\qquad\qquad + \vec{u}_i\otimes (A-\mu I)\vec{v}_i+ (A-\mu I) (\vec{v}_i \otimes \vec{u}_i) \big\| \\
&\quad\leq\frac{1}{2\sqrt{2}}  \big( 2\norm{(A-\lambda I)\vec{u}_i}\norm{\vec{v}_i} + 2 \norm{(A-\mu I)\vec{v}_i} \norm{\vec{u}_i} \big)\\
&\quad= \frac{1}{\sqrt{2}} \norm{(A-\lambda I)\vec{u}_i} + \frac{1}{\sqrt{2}}\norm{(A-\mu I)\vec{v}_i} \to 0.
\end{align}
Thus, $\frac{1}{2}(\lambda+\mu)\in \sigma_{\textrm{ap}}(A\odot I)$ and hence
\begin{equation}\label{eq:ApproxPointSum}
    \sigma_{\textrm{ap}}(A\odot I) \supseteq \tfrac{1}{2}(\sigma_{\textrm{ap}}(A)+\sigma_{\textrm{ap}}(A)).
\end{equation}

Recall that $\Omega(A)=\sigma(A)\setminus \sigma_{\textrm{ap}}(A)$ is a bounded open set.
Furthermore, \cite[Page 164]{Brown&Pearcy} shows that $\lambda \in \Omega(A)$ implies $\overline{\lambda}\in \sigma_{\mathrm{p}}(A^*)$. Since $\sigma(A)$ is closed and $\Omega(A) \subseteq \sigma(A)$, the boundary of $\Omega(A)$ is contained in $ \sigma(A) \setminus \Omega(A) = \sigma_{\textrm{ap}}(A)$.

Let $\lambda,\mu \in \sigma(A)$. Following \cite[Proof 2]{Brown&Pearcy}, we examine four special cases.
\begin{itemize}
    \item[(i)] If $\lambda,\mu \in \sigma_{\textrm{ap}}(A)$, then \eqref{eq:ApproxPointSum} ensures that 
    	$\frac{1}{2} (\lambda+\mu) \in \sigma_{\textrm{ap}}(A\odot I)\subseteq \sigma(A\odot I)$.

    \item[(ii)] If $\lambda,\mu \in \Omega(A)$, then    
    $\overline{\lambda},\overline{\mu} \in \sigma_{\textrm{p}}(A^*)\subseteq \sigma_{\textrm{ap}}(A^*)$. Then (i) ensures that 
    \begin{equation*}
        \tfrac{1}{2}(\overline{\lambda+\mu})\in \sigma_{\textrm{ap}}(A^* \odot I)
        =\sigma_{\textrm{ap}}((A\odot I)^*)\subseteq \sigma ((A\odot I)^*),
    \end{equation*}
    so $\frac{1}{2} (\lambda+\mu) \in \sigma(A\odot I)$.

    \item[(iii)] Suppose that $\lambda \in \sigma_{\textrm{ap}}(A)$ and $\mu\in \Omega(A)$.
    	Then $\overline{\lambda} \in \sigma(A^*)$ and $\overline{\mu}\in \sigma_{\textrm{ap}}(A^*)$. 
	If $\overline{\lambda}\in \sigma_{\textrm{ap}}(A^*)$, then (a) ensures that
	\begin{equation*}
            \tfrac{1}{2}(\overline{\lambda+\mu})\in \sigma_{\textrm{ap}}(A^*\odot I)\subseteq \sigma((A\odot I)^*),
            \quad \text{so $\tfrac{1}{2}(\lambda+\mu) \in \sigma(A\odot I)$}.
	\end{equation*}
	Suppose instead that $\overline{\lambda} \in \Omega(A^*)$.
	The openness of $\Omega(A)$ and $\Omega(A^*)$
            provide $\tau >0$ such that $\overline{\lambda}-t \in \Omega(A^*)$ and $\mu+t \in \Omega(A)$ for $0\leq t <\tau $.
            \begin{itemize}
                \item If $\overline{\lambda}-\tau \in \Omega(A^*)$ and $\mu+\tau  \in \sigma_{\textrm{ap}}(A)$,
                then $\lambda-\tau =\overline{\overline{\lambda}-\tau }\in \sigma_{\textrm{ap}}(A)$.
                Then (a) ensures that
                \begin{equation*}
                    \tfrac{1}{2}(\lambda+\mu) = \tfrac{1}{2}(\lambda-\tau +\mu+\tau  )\in \sigma_{\textrm{ap}}(A\odot I)\subseteq \sigma (A\odot I). 
                \end{equation*}

                \item If $\overline{\lambda}-\tau \in \sigma_{\textrm{ap}}(A^*)$ and $\mu+\tau  \in \Omega(A)$, this case is
                analogous to the previous one.

                \item Suppose that $\overline{\lambda}-\tau \in \sigma_{\textrm{ap}}(A^*)$ and $\mu+\tau  \in \sigma_{\textrm{ap}}(A)$.
                	 If $t_n \to \tau $ and $0 <t_n <\tau $, then $\overline{\lambda}-t_n \in \Omega(A^*)$ 
            	 and hence $\lambda-t_n \in \sigma_{\textrm{ap}}(A)$. Thus, 
            	 \begin{equation*}
                    	 \tfrac{1}{2}(\lambda-t_n+\mu+\tau )\in \sigma(A\odot I).
            	 \end{equation*}
                Since $\sigma(A \odot I)$ is closed, $\frac{1}{2}(\lambda+\mu) \in \sigma(A\odot I)$. 
                \end{itemize}

    \item[(iv)] The case $\lambda \in \Omega(A)$ and $\mu\in \sigma_{\textrm{ap}}(A)$ is analogous to (iii).
\end{itemize}
    In all cases
    $\frac{1}{2}(\lambda+\mu)\in \sigma(A \odot I)$,
so $\sigma(A\odot I)=\frac{1}{2}(\sigma(A)+\sigma(A))$.

\medskip\noindent(b) The proof is similar to that of (a), so we sketch the details.  
Lemma \ref{Proposition:SpectrumUnion} and Theorem \ref{Theorem:BrownPearcy} yield
$\sigma( A \odot A ) \subseteq \sigma \left( A \otimes A \right) =  \sigma(A) \sigma(A)$.
If $\lambda,\mu \in \sigma_{\textrm{ap}}(A)$, then an argument similar to that of (a) ensures that
$\lambda \mu \in \sigma_{\textrm{ap}}(A\odot A)$.  As above, we can follow \cite[Proof 2]{Brown&Pearcy}
and use a case-by-case analysis to show that $\lambda \mu\in \sigma(A \odot A)$, so that
$\sigma(A)\sigma(A) \subseteq \sigma(A \odot A)$.
\end{proof}

\section{Diagonal operators}\label{Section:Diagonal}
Since diagonal operators are among the most elementary operators one encounters 
in the infinite-dimensional setting \cite[Ch.~2]{OTBE}, it makes sense to consider their symmetric tensor products.
Let $\vec{e}_1,\vec{e}_2,\ldots$ be an orthonormal basis for $\h$ and suppose that
$L,M \in \B(\h)$ satisfy $L\vec{e}_i = \lambda_i \vec{e}_i$ and $M\vec{e}_i = \mu_i \vec{e}_i$ for $i \geq 1$.  For $i,j\geq 1$,
\begin{align}
&(L \odot M)(\vec{e}_i \odot \vec{e}_j) 
= \tfrac{1}{4}(L \otimes M + M \otimes L)( \vec{e}_i \otimes \vec{e}_j + \vec{e}_j \otimes \vec{e}_i)\\ 
&\qquad= \tfrac{1}{4}(L \vec{e}_i \otimes M  \vec{e}_j + M\vec{e}_i \otimes L\vec{e}_j+  L\vec{e}_j \otimes M\vec{e}_i + M\vec{e}_j \otimes L\vec{e}_i)\\
&\qquad= \tfrac{1}{4}(\lambda_i\vec{e}_i \otimes \mu_j \vec{e}_j+\mu_i \vec{e}_i \otimes \lambda_j\vec{e}_j+  \lambda_j\vec{e}_j \otimes \mu_i\vec{e}_i + \mu_j\vec{e}_j \otimes \lambda_i\vec{e}_i)\\
&\qquad= \tfrac{1}{2}(\lambda_i\mu_j+\lambda_j\mu_i) ( \vec{e}_i \odot \vec{e}_j ).\label{eq:Diagonal}
\end{align}
Thus, $L \odot M$ is a diagonal operator with  
\begin{equation}\label{eq:DiagonalSpectrum}
\sigma_{\textrm{p}}(L \odot M) 
= \left\{\tfrac{1}{2}(\lambda_i\mu_j+\lambda_j\mu_i ): i,j \geq 1 \right\}
\quad \text{and} \quad
\sigma(L \odot M)  = \sigma_{\textrm{p}}(L \odot M)^-.
\end{equation}

For symmetric products of diagonal operators, we can improve upon Theorem \ref{Theorem:Norm}.a.

\begin{Proposition}\label{Proposition:Diagonals}
Let $L,M$ be diagonal operators as above.
Then $ \|L\| \| M \| (\sqrt{2}-1) \leqslant \| L \odot M \| \leq \|L\| \|M\|$ and these inequalities are sharp.
\end{Proposition}

\begin{proof}
The computation \eqref{eq:Diagonal} shows that it suffices to prove the result for $2 \times 2$ diagonal matrices.
Let $L = \diag(\lambda_1,\lambda_2)$ and $M = \diag(\mu_1,\mu_2)$. 
By linearity we may assume $\norm{L}=\max\{ |\lambda_1|,|\lambda_2| \}=1$ and 
$\norm{M}=\max\{ |\mu_1|,|\mu_2|\}=1$,
Consider
$L \odot M$, which by \eqref{eq:2x2to3x3} we identify with 
$\diag( \lambda_1 \mu_1 , \frac{\lambda_1 \mu_2 + \lambda_2 \mu_1}{2}, \lambda_2 \mu_2)$.  Then
\begin{equation}\label{eq:SameReasoning}
\norm{L \odot M}=\max\big\{ | \lambda_1 \mu_1|,\tfrac{1}{2}|\lambda_1 \mu_2 + \lambda_2 \mu_1|,|\lambda_2 \mu_2|\big\}
\end{equation}
and one of the following holds:
\begin{enumerate}
    \item $|\lambda_1|=|\mu_1|=1$ or $|\lambda_2|=|\mu_2|=1$. 
    \item $|\lambda_1|=|\mu_2|=1$ or $|\lambda_2|=|\mu_1|=1$. 
\end{enumerate}
If (a) holds, then $\norm{L \odot M} \geq \max\{ |\lambda_1\mu_1|,|\lambda_2\mu_2| \}=1$. 
If (b) holds, then without loss of generality assume that $|\lambda_1|=|\mu_2|=1$. From \eqref{eq:SameReasoning},
\begin{align}
    \norm{L \odot M} 
    &\geqslant \inf_{|\mu_1|, |\lambda_2| \leqslant 1}  \max \big\{|\mu_1|, \tfrac{1}{2}|1+\lambda_2  \mu_1|,|\lambda_2|\big\}  \\
    &\geqslant \inf_{0 \leqslant s\leqslant 1}  \max \{s, \tfrac{1}{2}(1-  s^2)\}  
    = \sqrt{2} - 1.
\end{align}
The lower bound is attained by $L= \big[\begin{smallmatrix}
 1 & 0 \\ 0  & \sqrt{2}-1\end{smallmatrix}\big]$ and $M= \big[ \begin{smallmatrix}-\sqrt{2}+1& 0 \\ 0&1\end{smallmatrix}\big]$. 
The upper bound is attained by $L=M=I$.
\end{proof}

\begin{Proposition}
\begin{enumerate}
\item There exists a selfadjoint diagonal operator $D$ such that $\sigma(D)$ has measure zero in $\R$
and $\sigma(D\odot D)$ has positive measure in $\R$.

\item There exists a diagonal operator $D$ such that $\sigma(D)$ has planar Lebesgue measure zero
and $\sigma(D\odot D)$ has positive planar Lebesgue measure.
\end{enumerate}
\end{Proposition}

\begin{proof}
\noindent(a) Let $\CC$ denote the Cantor set, which has measure zero. The exponential function is differentiable, so by \cite[Lemma 7.25]{RudinRCA}
$\{ e^\mu : \mu \in \CC\cap \Q \}^- = \{e^c : c \in \CC\}$ has measure zero in $\R$.
Let $D$ be a diagonal operator with point spectrum $\{ e^{\lambda} : \lambda \in \CC \cap \Q\}$,
so that $\sigma(D) = \{e^c : c \in \mathscr{C}\}$ has measure zero.  
Since $\mathscr{C}+\mathscr{C}=[0,2]$,
Theorem \ref{Theorem:SpectrumSame}.b ensures that
$\sigma(D \odot D) = \sigma(D)\sigma(D) =  \{ e^{c+d} : c,d \in \mathscr{C}\} = [1,e^2]$
has positive measure in $\R$.

\medskip\noindent(b) Let $D$ be a diagonal operator with point spectrum
$\{ e^{\lambda + i \mu} : \lambda \in \CC \cap \Q, \mu \in \Q \cap [0,2\pi)\}$.
Then $\sigma(D)$ has planar measure zero, but an argument similar to that in (a)
ensures that $\sigma(D\odot D)$ is the annulus centered 
at $0$ with radii $1$ and $e^2$, which has positive measure.
\end{proof} 

Below the brackets $\{\hspace{-3.5pt}\{$ and $\}\hspace{-3.5pt}\}$ indicate a \emph{multiset};
that is, a set that permits multiplicity.

\begin{Proposition}\label{Proposition:MultiDiagSpec}
Let $A_1, A_2,  \ldots, A_n \in \mathcal{B} (\h)$ be commuting diagonal operators with
$\sigma_{\textrm{p}}(A_i) = \{\hspace{-3.5pt}\{ \lambda_1^{(i)}, \lambda_2^{(i)},\ldots \}\hspace{-3.5pt}\}$ allowing for repetition. Then
\begin{equation*}\small
\sigma_p(A_1 \odot A_2 \odot \cdots \odot A_n) 
=  \bigg\{\!\!\!\bigg\{\frac{1}{n!} \sum_{\pi \in \Sigma_n }\lambda_{i_1}^{(\pi(1))} \lambda_{i_2}^{(\pi(2))}\cdots \lambda_{i_n}^{(\pi(n))}  : 
 i_1 \leq i_2 \leq \ldots \leq i_n   \bigg\}\!\!\!\bigg\}
\end{equation*}
and $\sigma(A_1 \odot A_2 \odot \cdots \odot A_n) = \sigma_p(A_1 \odot A_2 \odot \cdots \odot A_n)^-$.
\end{Proposition}

\section{The shift operator and its adjoint}\label{Chapter:ShiftAndAdjoint}
In this section we find the spectrum of the symmetric tensor product of the unilateral shift and its adjoint.
Let $(Sf)(z) = zf(z)$ denote the unilateral shift on $H^2(\D)$ \cite[Ch.~5]{OTBE}.
Its adjoint is the backward shift $(S^*f)(z) = (f(z)-f(0))/z$.

\begin{Theorem}\label{Theorem:ShiftSpectra}
The selfadjoint operators $S \odot S^*$ and $S \wedge S^*$ satisfy
\begin{equation*}
\sigma_{\textrm{p}}(S \odot S^*)
=\big\{\hspace{-4.5pt}\big\{ \cos\! \big(  \tfrac{(2j-1) \pi} { k+2} \big)  : \text{$k \geq 0$ and $1\leq  j \leq  \lfloor \tfrac{k+2}{2} \rfloor$} \big\}\hspace{-4.5pt}\big\} 
\end{equation*}
and
\begin{equation*}
\sigma_{\textrm{p}}(S \wedge S^*)
=\big\{\hspace{-4.5pt}\big\{ \cos\! \big(  \tfrac{2j \pi} { k+2} \big)  : \text{$k \geq 1$ and $1\leq  j \leq  \lfloor \tfrac{k+1}{2} \rfloor$} \big\}\hspace{-4.5pt}\big\} ,
\end{equation*}
with the eigenvalues in these multisets repeated by multiplicity.  Moreover,
$\sigma (S \odot S^*) = \sigma_{\textrm{ap}} (S \odot S^*) =  \sigma (S \wedge S^*) = \sigma_{\textrm{ap}} (S \wedge S^*)  =  [-1,1]$
and $\norm{S \odot S^*} = \norm{S \wedge S^*}=1$.
\end{Theorem}

\begin{proof}
Identify $H^2(\D) \otimes H^2(\D)$ with $H^2(\D^2)$ as in Example \ref{Example:Hardy} and consider
\begin{equation*}
T = \frac{1}{2}(S\otimes S^*+S^* \otimes S)(z^i   w^j)
=
\begin{cases}
\frac{1}{2}(z^{i+1}   w^{j-1}+z^{i-1}   w^{j+1}) & \text{if $i,j \geq 1$},\\[2pt]
\frac{1}{2}z^{i+1}   w^{j-1} & \text{if $i = 0$ and $j \geq 1$},\\[2pt]
\frac{1}{2}z^{i-1}   w^{j+1} & \text{if $i \geq1$ and $j =0$},\\[2pt]
0 & \text{if $i=j=0$}.
\end{cases}
\end{equation*}
Define $\V_0 = \V_0^+ = \operatorname{span}\{1\}$ and $\V_0^- = \{ 0\}$.
For $k\geq 1$, let
\begin{align*}
\V_k &= \operatorname{span}\{ z^i w^{k-i} : 0 \leq i \leq k\}, && \text{so $\dim \V_k = k+1$}, \\
\V_{k}^+  &= \operatorname{span}\{z^i w^{k-i} + z^{k-i} w^i: 0 \leq i \leq {\lfloor \tfrac{k}{2}\rfloor}\}, && \text{so $\dim \V_k^+ = \lfloor \tfrac{k}{2} \rfloor + 1$},\\
\V_{k}^- &= \operatorname{span}\{ z^i w^{k-i} - z^{k-i}w^i : 0 \leq i \leq {\lfloor \tfrac{k-1}{2}\rfloor}\}, && \text{so $\dim \V_k^- = \lfloor \tfrac{k-1}{2}\rfloor+1$},
\end{align*}
and note that 
$\dim \V_k =   \dim \V_k^+ + \dim \V_k^-$ for $k\geq 1$
by a parity argument (or Proposition \ref{Proposition:TwoSum}).  Recall from \eqref{eq:H2DirectSum} that
$H^2(\D^2) = H^2_{\operatorname{sym}}(\D^2) \oplus H^2_{\operatorname{asym}}(\D^2)$ is an orthogonal direct sum.
We have $\V_k = \V_k^+ \oplus \V_k^-$ for $k \geq 1$,
in which each $\V_k, \V_{k}^+ , \V_{k}^- $ is $T$-invariant, and
\begin{equation}\label{Eq:directsumspaces}
    H^2(\D^2) = \bigoplus_{k=0}^{\infty} \V_k, \qquad H^2_{\operatorname{sym}}(\D^2) = \bigoplus_{k=0}^{\infty} \V_{k}^+ , \qquad 
    H^2_{\operatorname{asym}}(\D^2) = \bigoplus_{k=1}^{\infty} \V_{k}^- .
\end{equation} 

With respect to the orthonormal basis $\{z^{k-i} w^i\}_{i=0}^k$ of $\V_k$, we identify 
the restriction $T|_{\V_k}$ with the $(k+1) \times (k+1)$ matrix (by convention $A_0 = [0]$)
\begin{equation*}
A_k =\tiny \begin{bmatrix}
0 & \frac{1}{2} &  &  &  & \\
\frac{1}{2} & 0 & \frac{1}{2} &  &  &\\[-5pt]
 & \frac{1}{2} & 0 & \ddots &  &\\
 &  & \ddots & 0 &  \frac{1}{2} &\\
 &  &  & \frac{1}{2} & 0 &\frac{1}{2}\\
  &  &  &  & \frac{1}{2} & 0
\end{bmatrix}.
\end{equation*}
From \cite[Prop.~2.1]{tridiageigen}, we have
\begin{equation}\label{eq:SpectrumAk}
\sigma (A_k) = \big\{ \cos \! \big( \tfrac{j \pi}{  k+2} \big) : j = 1,2, \ldots,k+1 \big\}.
\end{equation}

For $k$ odd, with respect to the orthonormal basis
$\{\frac{1}{\sqrt{2}}(z^{k-i} w^i+z^i w^{k-i})\}_{i=0}^{ \frac{k-1}{2}}$  of $\V_k^+$,
we identify the restriction $T|_{\V_k^+}$ with the $\frac{k+1}{2}\times \frac{k+1}{2}$ matrix
\begin{equation*}
B_k =\tiny \begin{bmatrix}
0 & \frac{1}{2} &  &  &  & \\
\frac{1}{2} & 0 & \frac{1}{2} &  &  &\\[-5pt]
 & \frac{1}{2} & 0 & \ddots &  &\\
 &  & \ddots & 0 &  \frac{1}{2} &\\
 &  &  & \frac{1}{2} & 0 & \frac{1}{2}\\[2pt]
  &  &  &  & \frac{1}{2} & \frac{1}{2}
\end{bmatrix}.
\end{equation*}
For $k$ even, with respect to the orthonormal basis
$\{\frac{1}{\sqrt{2}}(z^{k-i} w^i+z^i w^{k-i})\}_{i=0}^{ \frac{k}{2}-1}\cup \{z^{k/2}w^{k/2}\}$  of $\V_k^+$,
we identify the restriction $T|_{\V_k^+}$ with the $\left(\frac{k}{2}+1\right)\times \left(\frac{k}{2}+1\right)$ matrix
\begin{equation*}
B_k = \tiny\begin{bmatrix}
0 & \frac{1}{2} &  &  &  & \\
\frac{1}{2} & 0 & \frac{1}{2} &  &  &\\[-5pt]
 & \frac{1}{2} & 0 & \ddots &  &\\
 &  & \ddots & 0 &  \frac{1}{2} &\\
 &  &  & \frac{1}{2} & 0 & \frac{1}{\sqrt{2}}\\
  &  &  &  & \frac{1}{\sqrt{2}} & 0
\end{bmatrix} ,
\end{equation*}
with the convention $B_0 = [0]$.  We identify the spectrum of the $B_k$ later.

With respect to the orthonormal basis $\{\frac{1}{\sqrt{2}}(z^i w^{k-i} - z^{k-i} w^i)\}_{i=0}^{\lfloor \frac{k-1}{2}\rfloor}$ of
$\V_k^-$, we identify the restriction $T|_{ \V_{k}^- }$ with the $\lfloor \frac{k+1}{2}\rfloor \times \lfloor \frac{k+1}{2}\rfloor$
matrix (note that $C_1 = [- \frac{1}{2}]$ and $C_2 = [0]$)
\begin{equation*}
C_k = \tiny
    \begin{bmatrix}
0 & \frac{1}{2} &  &  &  & \\
\frac{1}{2} & 0 & \frac{1}{2} &  &  &\\[-5pt]
 & \frac{1}{2} & 0 & \ddots &  &\\
 &  & \ddots & 0 &  \frac{1}{2} &\\
 &  &  & \frac{1}{2} & 0 &\frac{1}{2}\\[1pt]
  &  &  &  &\frac{1}{2} & \frac{(-1)^k - 1}{4}
\end{bmatrix}.
\end{equation*}
For $k\geq 2$ even, 
$\sigma (C_k) =  \{ \cos(\frac{2j \pi}{  k+2} )  : j = 1,2, \ldots, \frac{k}{2} \}$ \cite[Prop.~2.1]{tridiageigen}.
Suppose $k\geq 1$ is odd.
By \cite[eq.~11]{tridiageigen}, $\lambda \in \sigma( 2 C_k)$ if and only if $\lambda =  -2 x $, where 
$x\in[-1,1]$ solves
\begin{equation*}
    (-1 + 2x) \frac{\sin (\frac{1}{2}(k+1) \cos ^{-1}(x))}{\sin (\cos ^{-1} (x))} - \frac{\sin (\frac{1}{2}(k-1) \cos ^{-1}(x))}{\sin (\cos ^{-1} (x))}  = 0.
\end{equation*}
Since $\cos(\frac{(2\ell -1 )\pi}{k+2})$ for $\ell = 1,2, \ldots,\frac{k+1}{2}$
are the distinct solutions to this equation,
$\sigma( C_k) =\{  - \cos(\frac{(2\ell - 1)\pi}{k+2}) : \ell = 1,2,  \ldots,\frac{k+1}{2} \}$.
Since $- \cos{(x)} = \cos{ (\pi - x)}$, we can reindex and rewrite this as
$\sigma (C_k) =   \{ \cos(\tfrac{2j \pi}{k+2}) : j = 1,2,  \ldots,\tfrac{k+1}{2} \}$.
Regardless of the parity of $k$, 
\begin{equation}\label{eq:SpectrumCk}
    \sigma (C_k) = \big\{ \cos \! \big(  \tfrac{2j \pi} { k+2} \big)  : j = 1,2, \ldots, \big\lfloor \tfrac{k+1}{2} \big\rfloor \big\}.
\end{equation}

Since $\mathcal{V}_k=\V_{k}^+ \oplus\V_{k}^- $, up to unitary equivalence $A_k = B_k \oplus C_k$.  Thus,
\begin{equation}\label{eq:SpectrumTripleUnion}
    \sigma (A_k) = \sigma (B_k) \cup \sigma (C_k).
\end{equation}
From \eqref{eq:SpectrumAk}, \eqref{eq:SpectrumCk}, and \eqref{eq:SpectrumTripleUnion}, we obtain
\begin{equation}\label{specBk}
    \sigma (B_k) = \big\{ \cos\! \big(  \tfrac{(2j-1) \pi} { k+2} \big)  : j = 1,2, \ldots, \big\lfloor \tfrac{k+2}{2} \big\rfloor \big\}. 
\end{equation}

Since $S \odot S^*$ and $S \wedge S^*$ are selfadjoint and have norm at most $1$,
their spectra are contained in $[-1,1]$.
Up to unitary equivalence, \eqref{eq:H2DirectSum} and \eqref{Eq:directsumspaces} imply that
\begin{equation}\label{Eq:directsum}
S \odot S^* = \bigoplus_{k=0}^{\infty} B_k \quad \text{and} \quad S \wedge S^* = \bigoplus_{k=1}^{\infty} C_k .
\end{equation}
This yields the claimed point spectra of $S \odot S^*$ and $S \wedge S^*$.  
A density argument reveals that $[-1,1] = \sigma_{\textrm{p}}(S \odot S^*)^- \subseteq \sigma_{\textrm{ap}}(S \odot S^*) \subseteq \sigma(S \odot S^*) \subseteq [-1,1]$, so equality holds throughout.  A similar argument treats $S \wedge S^*$.  
\end{proof}

\section{Shifts and diagonal operators}\label{Chapter:ShiftDiagonal}

We consider here the symmetric tensor product of shift operators and diagonal operators.
This setting suggests working on the sequence space $\ell^2$ instead of $H^2(\D)$ \cite[Sec.~1.2]{OTBE}.
Let $\vec{e}_0, \vec{e}_1, \ldots$ be the standard basis of $\ell^2$ and consider
the unilateral shift $S \vec{e}_i = \vec{e}_{i+1}$ \cite[Ch.~3]{OTBE}.  Its adjoint is given by
$S^* \vec{e}_i = \vec{e}_{i-1}$ for $i \geq 1$ and $S^* \vec{e}_0 = 0$.

\begin{Theorem}\label{Theorem:ShiftDiagSpec}
Let $M = \diag(\mu_0,\mu_1,\ldots)$ be a bounded diagonal operator on $\ell^2$.
\begin{enumerate}
\item $\frac{1}{\sqrt{2}} \norm{M} \leq \norm{S \odot M} \leq \| M \|$. Both inequalities are sharp.

\item If some $\mu_i = 0$ or the set of nonzero $\mu_i$ is bounded away from $0$, then $0 \in \sigma_{\mathrm{p}}(S \odot M)$.

\item $\sigma_{\mathrm{p}} (S \odot M) \subseteq \{ 0 \}$. 
\end{enumerate}
\end{Theorem}

\begin{proof}
(a) Since $M^* = \diag(\overline{\mu_0},\overline{\mu_1},\ldots)$ is a diagonal operator
and  $(S \odot M^*)^* = S^* \odot M$ by Proposition \ref{Proposition:Adjoint}, this follows from 
Theorem \ref{Theorem:DiagBackShift}.a below.

\medskip\noindent(b) 
Suppose that the set of nonzero $\mu_i$ is bounded away from zero. Note that for all $i,j \geq 0$,
\begin{equation}\label{eq:SDiag}
(S \odot M)({\vec{e}_i} \odot {\vec{e}_j}) 
= \frac{\mu_j}{2}{\vec{e}_{i+1}} \odot {\vec{e}_j}+\frac{\mu_i}{2}{\vec{e}_i} \odot {\vec{e}_{j+1}}.
\end{equation}
If some $\mu_i = 0$, then \eqref{eq:SDiag} ensures that $0 \in \sigma_{\mathrm{p}}(S \odot M)$
since $(S \odot M)( \vec{e}_i \odot \vec{e}_i) = 0 $. 
Thus, we may assume that $|\mu_i| \geq \delta > 0$ for all $i\geq 0$.  Define $C = \sqrt{\norm{M}/\delta}$.

Let $\sum_{i \leq j } | a_{ij}|^2 < \infty$ and let 
$\vec{v} = 2\sum_{0 \leq i \leq j< \infty} a_{ij} \vec{e}_i \odot \vec{e}_j$,
which is well defined by Lemma \ref{Lemma:SquareSumConverge}. Then \eqref{eq:SDiag} ensures that
\begin{equation}\label{eq:SMeigen}
    (S \odot M) \vec{v} = \sum_{0 \leq i \leq j< \infty} a_{ij}\left( \mu_j \vec{e}_{i+1}\odot \vec{e}_j + \mu_i \vec{e}_i \odot \vec{e}_{j+1} \right).
\end{equation}
When \eqref{eq:SMeigen} is expanded, the coefficient of $\vec{e}_k \odot \vec{e}_{\ell}$ for $k \leq \ell$ is
\begin{align}
	&0 && \text{if $k = \ell = 0$}, \\
    & 2\mu_0 a_{0, 0}  &&\text{if $k=0$ and $\ell = 1$}, \label{eq:Kernel0}\\
    & \mu_0 a_{0, \ell-1}  &&\text{if $k=0$ and $\ell \geq 2$}, \label{eq:Kernel1}\\
    &\mu_k a_{k-1, k}   &&\text{if $1 \leq k=\ell$}, \label{eq:Kernel2}\\
     &2\mu_k a_{k, k}  + \mu_{k+1} a_{k-1, k+1}  &&\text{if $k \geq 1$ and $\ell = k+1$}, \label{eq:Kernel3}\\
    &\mu_{k} a_{k, \ell-1}  + \mu_{\ell} a_{k-1, \ell}  &&\text{if $k \geq 1$ and $\ell \geq k+2$}.\label{eq:Kernel4}
\end{align}
Then $(S \odot M)\vec{v} = \vec{0}$ if and only if 
the $a_{k,\ell}$ are square summable and \eqref{eq:Kernel1} - \eqref{eq:Kernel4} vanish for all $\ell \geq k \geq 0$.  We define such $a_{k,\ell}$, not all zero, in four steps; see
Figure \ref{Figure:BigMatrix}.

\begin{figure}
\begin{equation*}\footnotesize
\begin{bmatrix}
\color{violet} a_{0,0} & \color{violet}a_{0,1} & \color{violet}a_{0,2} & \color{violet}a_{0,3} & \color{violet}a_{0,4} & \color{violet}a_{0,5} & \color{violet}a_{0,6} \color{violet}& \color{violet}a_{0,7}  &\color{violet}a_{0,8} \\
 \color{lightgray}a_{1,0} & \color{green!50!black}a_{1,1} & \color{red}a_{1,2} & \color{green!50!black}a_{1,3} & \color{red}a_{1,4} & \color{blue!80!black}a_{1,5} & \color{red}a_{1,6} & \color{blue!80!black}a_{1,7} & \color{red}a_{1,8}\\
 \color{lightgray}a_{2,0} & \color{lightgray}a_{2,1} & \color{green!50!black}a_{2,2} & \color{red}a_{2,3} & \color{green!50!black}a_{2,4} & \color{red}a_{2,5} & \color{blue!80!black}a_{2,6} & \color{red}a_{2,7} & \color{blue!80!black}a_{2,8}\\
\color{lightgray} a_{3,0} & \color{lightgray}a_{3,1} & \color{lightgray}a_{3,2} & \color{green!50!black}a_{3,3} & \color{red}a_{3,4} & \color{green!50!black}a_{3,5} & \color{red}a_{3,6} & \color{blue!80!black}a_{3,7} & \color{red}a_{3,8}\\
\color{lightgray} a_{4,0} & \color{lightgray}a_{4,1} & \color{lightgray}a_{4,2} &\color{lightgray} a_{4,3} & \color{green!50!black}a_{4,4} & \color{red}a_{4,5} & \color{green!50!black}a_{4,6} & \color{red}a_{4,7} & \color{blue!80!black}a_{4,8}\\
\color{lightgray} a_{5,0} & \color{lightgray}a_{5,1} & \color{lightgray}a_{5,2} & \color{lightgray}a_{5,3} & \color{lightgray}a_{5,4} & \color{green!50!black}a_{5,5} & \color{red}a_{5,6} & \color{green!50!black}a_{5,7} & \color{red}a_{5,8}\\
\color{lightgray} a_{6,0} & \color{lightgray}a_{6,1} & \color{lightgray}a_{6,2} & \color{lightgray}a_{6,3} & \color{lightgray}a_{6,4} & \color{lightgray}a_{6,5} & \color{green!50!black}a_{6,6} & \color{red}a_{6,7} & \color{green!50!black} a_{6,8}\\
\color{lightgray} a_{7,0} & \color{lightgray}a_{7,1} & \color{lightgray}a_{7,2} & \color{lightgray}a_{7,3} & \color{lightgray}a_{7,4} & \color{lightgray}a_{7,5} & \color{lightgray}a_{7,6} & \color{green!50!black}a_{7,7}  & \color{red}a_{7,8}\\
\color{lightgray} a_{8,0} & \color{lightgray}a_{8,1} & \color{lightgray}a_{8,2} & \color{lightgray}a_{8,3} & \color{lightgray}a_{8,4} & \color{lightgray}a_{8,5} & \color{lightgray}a_{8,6} & \color{lightgray}a_{8,7}  & \color{green!50!black} a_{8,8}\\
\end{bmatrix}
\end{equation*}
\caption{Colors denote the step where the $a_{k,\ell}$ are fixed:
Step (1) is in {\color{violet}violet};
(2) is in {\color{red}red};
(3) is in {\color{green!50!black}green}; and
(4) is in {\color{blue!80!black}blue}.
The symmetry of symmetric tensors permits us to focus on $\ell \geq k \geq 0$.
The {\color{violet}violet} and {\color{red}red} values are zero.}
\label{Figure:BigMatrix}
\end{figure}

\begin{enumerate}
\item[(1)] Let $a_{0,0} = 0$.  For $\ell \geq 1$, let $a_{0,\ell-1} = 0$ so that \eqref{eq:Kernel1} vanishes.

\item[(2)] For each $k \geq 2$, let 
$a_{k-1,k} = a_{k-1,k+2} = a_{k-1,k+4} = \cdots = 0$.
Then \eqref{eq:Kernel2} and \eqref{eq:Kernel4} vanish for $k \geq 2$ and even $\ell \geq k+2$.

\item[(3)] Let $a_{1,1}=0$ and, for each $k \geq 2$, let $a_{k,k}$ and $a_{k-1,k+1}$ be such that
\begin{equation}\label{eq:PairOrthog}
\begin{bmatrix}
2a_{k,k} \\
a_{k-1,k+1}
\end{bmatrix} 
\perp
\begin{bmatrix}
\overline{\mu_k} \\
\overline{\mu_{k+1}}
\end{bmatrix}
\quad \text{and} \quad 
0<
\left\| 
\begin{bmatrix}
2a_{k,k} \\
a_{k-1,k+1}
\end{bmatrix} \right \| < \frac{1}{(k+1)^{3/2}}.
\end{equation}
Then \eqref{eq:Kernel3} vanishes for $k \geq 2$ and $\ell = k+1$.

\item[(4)] For $k \geq 2$, let
\begin{equation}\label{eq:Recur}
    a_{k-1, k+3} = - a_{k, k+2} \frac{\mu_k}{\mu_{k+3}}, \qquad
    a_{k-1, k+5}= - a_{k, k+4} \frac{\mu_k}{\mu_{k+5}},\ldots.  
\end{equation}
Then \eqref{eq:Kernel4} vanishes for all $k \geq 1$ with odd $\ell \geq k+3$.
\end{enumerate}
This completes the definition of the $a_{k,\ell}$.  We must prove that they are square summable.

For $k \geq 1$, \eqref{eq:PairOrthog} yields
\begin{equation}\label{eq:CubeDecay}
    |a_{k,k}|^2 < \frac{1}{(k+1)^3} \quad \text{and} \quad |a_{k-1,k+1}|^2 < \frac{1}{(k+1)^3}.
\end{equation}
Then \eqref{eq:Recur}, and then \eqref{eq:CubeDecay} with $k+1$ in place of $k$
ensure that 
\begin{equation}\label{eq:Secondn}
    |a_{k-1,k+3}|^2 
    = |a_{k, k+2}|^2 \left| \frac{\mu_{k}}{\mu_{k+3}} \right|^2 
    \leq C|a_{k, k+2}|^2
    \leq \frac{C}{(k+2)^3}
\end{equation}
for $k \geq 1$.
Next \eqref{eq:Recur} and then \eqref{eq:Secondn} with $k+1$ in place of $k$ imply that
\begin{equation*}
|a_{k-1,k+5}|^2 
= |a_{k, k+4}|^2 \left| \frac{\mu_{k}}{\mu_{k+5}} \right|^2 
\leq C|a_{k, k+4}|^2 
\leq \frac{C}{(k+3)^3},
\end{equation*}
so induction yields
\begin{equation}\label{eq:Boundkr}
| a_{k,k+2r} |^2 \leq \frac{C}{(k + r)^3}.
\end{equation}
Then Step 1, Step 2, and \eqref{eq:Boundkr} ensure that
\begin{equation*}
\sum_{0\leq k \leq \ell} |a_{k,\ell}|^2
= \sum_{k=1}^{\infty} \sum_{\ell \geq k} |a_{k,\ell}|^2 
= \sum_{k=1}^{\infty} \sum_{r=0}^{\infty} |a_{k,k+2r}|^2 
\leq C \sum_{k=1}^{\infty} \sum_{r=0}^{\infty} \frac{1}{(k + r)^3} ,
\end{equation*}
which is finite by a standard argument in the study of elliptic functions \cite[Prop.~10.4.2]{SimonComplex}.\footnote{The double sum can be explicitly evaluated. Write the summands in an array with $r$ indexing the columns and $k$ the rows.  Sum
each column and simplify to reduce the double sum to the well-known $\sum_{k=1}^{\infty} \frac{1}{k^2} = \frac{\pi^2}{6}$.}
Thus, $\vec{v}$ is a well-defined vector in the kernel of $S \odot M$.

\medskip\noindent(c) Suppose that $\lambda \neq 0$ and $(S \odot M) \vec{v} = \lambda \vec{v}$, in which
$\vec{v} = 2\sum_{0 \leq i \leq j< \infty} a_{ij} \vec{e}_i \odot \vec{e}_j$ and
$\sum_{0\leq i \leq j < \infty} | a_{ij}|^2 < \infty$.  Then \eqref{eq:SDiag} ensures that
\begin{align}
\vec{0} 
&= ((S \odot M) -\lambda I ) \vec{v}
= 2\sum_{0 \leq i \leq j< \infty} a_{i,j} ((S \odot M) -\lambda I )(\vec{e}_i\odot \vec{e}_j) \\
&=\!\!\sum_{0 \leq i \leq j< \infty} \!\!\!\! \!\! a_{i,j}\mu_j \vec{e}_{i+1}\odot{\vec{e}_j}
+ \!\!\!\!  \sum_{0 \leq i \leq j< \infty}  \!\!\!\!\!\!  a_{i,j}\mu_i \vec{e}_i\odot \vec{e}_{j+1} 
-  \!\!\!\!  \sum_{0 \leq i \leq j< \infty}  \!\!\!\! \!\! \lambda a_{i,j}{\vec{e}_i}\odot{\vec{e}_j} .\label{eq:Coeff4SM}
\end{align}
When \eqref{eq:Coeff4SM} is expanded, the coefficient of $\vec{e}_k \odot \vec{e}_{\ell}$ for $k \leq \ell$ is
\begin{align}
	0=&-\lambda a_{0,0}&& \text{if $k = \ell = 0$}, \label{eq:Eigen0}\\
    0=& 2\mu_0 a_{0, 0} -\lambda a_{0,1} &&\text{if $k=0$ and $\ell = 1$}, \label{eq:EigenX}\\
    0=& \mu_0 a_{0, \ell-1} -\lambda a_{0,\ell} &&\text{if $k=0$ and $\ell \geq 2$}, \label{eq:Eigen1}\\
    0=&\mu_k a_{k-1, k} - \lambda a_{k,k}  &&\text{if $1 \leq k=\ell$}, \label{eq:Eigen2}\\
     0=&2\mu_k a_{k, k}  + \mu_{k+1} a_{k-1, k+1}  -\lambda a_{k,k+1}&&\text{if $k \geq 1$ and $\ell = k+1$}, \label{eq:Eigen3}\\
    0=&\mu_{k} a_{k, \ell-1}  + \mu_{\ell} a_{k-1, \ell} - \lambda a_{k,\ell} &&\text{if $k \geq 1$ and $\ell \geq k+2$}.\label{eq:Eigen4}
\end{align}
We use induction to prove that $a_{k,\ell} = 0$ for $0 \leq k \leq \ell< \infty$.
For $k \geq0$, let $\texttt{P}(k)$ be the statement ``$a_{k, k+i} = 0$ for all $i \geq 0$.'' 
The truth of $\texttt{P}(0)$ follows from \eqref{eq:Eigen0}, which ensures that $a_{0,0}=0$, and
induction on $\ell$ using \eqref{eq:Eigen1}, which yields $a_{0,\ell} = 0$ for $\ell \geq 0$.

Suppose $\texttt{P}(k-1)$ is true: $a_{k-1,\ell} = 0$ for $\ell \geq k-1$.
Then \eqref{eq:Eigen2} yields
$\lambda a_{k,k} = \mu_{k} a_{k-1,k} = 0$, so $a_{k,k}=0$.
Next \eqref{eq:Eigen3} ensures that
$\lambda a_{k,k+1} =2 \mu_{k} a_{k,k}  + \mu_{k+1} a_{k-1,k+1} = 0$, so $a_{k,k+1}=0$.
Finally, \eqref{eq:Eigen4} and induction on $\ell$ tell us that
$\lambda a_{k,\ell} = \mu_{k} a_{k, \ell-1}  + \mu_{\ell} a_{k-1, \ell} = 0$ for $\ell \geq k+2$.
Thus, $\texttt{P}(k)$ is true, so $\vec{v} = \vec{0}$ and $\lambda \notin \sigma_{\textrm{p}}(S \odot M)$.
\end{proof}


\begin{Theorem}\label{Theorem:DiagBackShift}
Let $M = \diag(\mu_0,\mu_1,\ldots)$ be a bounded diagonal operator on $\ell^2$.
\begin{enumerate}
\item $\frac{1}{\sqrt{2}} \norm{M} \leq \norm{S^* \odot M} \leq \| M \|$. Both inequalities are sharp.
\item $\{ |z| < \frac{1}{2}|\mu_0| \} \cup \{0\}\subset \sigma_{\textrm{p}} ( S^* \odot M ) $.
\end{enumerate}
\end{Theorem}

\begin{proof}
(a) Since $\|S^*\| = 1$, Theorem \ref{Proposition:BasicNormInequality} yields $\norm{S^* \odot M} \leq \| M \|$.
Equality holds for $M= I$ because $\sigma(S^*) = \D^-$ \cite[Prop.~5.2.4.a]{OTBE} and
$\sigma(S^*\odot I) = \frac{1}{2}(\sigma(S^*)+\sigma(S^*)) \subseteq \D^-$
by Theorem \ref{Theorem:SpectrumSame}.  Thus, $\norm{S^* \odot I} \geq 1 = \norm{I}$.

Suppose that $M \vec{e}_i = \mu_i \vec{e}_i$ for $i\geq 0$.  Then
\begin{equation}\label{eq:SAdjointD}
    (S^* \odot M) (\vec{e}_i \odot \vec{e}_j)
    =
    \begin{cases}
    \frac{1}{2}( \mu_j \vec{e}_{i-1}\odot \vec{e}_j+\mu_i \vec{e}_i\odot \vec{e}_{j-1} ) & \text{if $i,j\neq 0$},\\[3pt]
    \frac{1}{2}( \mu_i \vec{e}_i\odot \vec{e}_{j-1} ) & \text{if $0 = i < j$},\\[3pt]
    \frac{1}{2}( \mu_j \vec{e}_{i-1}\odot \vec{e}_j  ) & \text{if $0=j<i$},\\[3pt]
     \vec{0} & \text{if $i=j=0$}.
    \end{cases}
\end{equation}
For each $\epsilon>0$, there is a $\mu_i$ such that
$\norm{M}-\epsilon \leq |\mu_i|$. 
Then \eqref{eq:SAdjointD} ensures that
\begin{equation}\label{eq:CompAbove}
 \frac{\norm{M}-\epsilon}{\sqrt{2}}\leq     \frac{| \mu_i |}{\sqrt{2}} =\| \mu_i \vec{e}_{i-1} \odot \vec{e}_{i} \| = \| (S^*\odot M)( \vec{e}_{i} \odot \vec{e}_{i}) \| \leqslant  \norm{S^* \odot M}.
\end{equation}
Let $\epsilon\to 0$ to obtain the desired lower bound.

If $\mu_i = \delta_{i0}$ for all $i\geq 0$, then 
$\norm{(S^* \odot M)(\sqrt{2} \vec{e}_0 \odot \vec{e}_1)}=\norm{\frac{1}{\sqrt{2}}(  \vec{e}_0 \odot \vec{e}_0 )} = \frac{1}{\sqrt{2}}$ and $\norm{M}=1$, so the lower bound is sharp.

\medskip\noindent(b) If $\mu_0 = 0$, the last line of \eqref{eq:SAdjointD} ensures that $0 \in \sigma_{\textrm{p}}(S^* \odot M)$.
Let $\mu_0 \neq 0$ and $| \lambda | < \frac{1}{2}|\mu_0|$.  
Lemma \ref{Lemma:SquareSumConverge} permits us to define
$\vec{v} = \sum_{j=0}^{\infty} \frac{(2\lambda)^j}{\mu_0^j} \vec{e}_0 \odot \vec{e}_j$.
Then \eqref{eq:SAdjointD} ensures that $\lambda \in \sigma_{\textrm{p}}(S^*\odot M)$ since
\begin{align*}
(S^* \odot M)\vec{v}
&= (S^* \odot M)\Big( \sum_{j=0}^{\infty} \frac{(2\lambda)^j}{\mu_0^j} \vec{e}_0 \odot \vec{e}_j \Big) 
=  \sum_{j=0}^{\infty} \frac{(2\lambda)^j}{\mu_0^j} (S^* \odot M) (\vec{e}_0 \odot \vec{e}_j ) \\
&=\frac{1}{2} \sum_{j=0}^{\infty} \frac{(2\lambda)^j}{\mu_0^j}  \mu_0 \vec{e}_0 \odot \vec{e}_{j-1} 
= \lambda \sum_{j=1}^{\infty} \frac{(2\lambda)^{j-1}}{\mu_0^{j-1}}   \vec{e}_0 \odot \vec{e}_{j-1} 
= \lambda \vec{v}. \qedhere
\end{align*}
\end{proof}

\section{Questions for further research}\label{Section:Questions}
We conclude with questions to spur future research.  Some are general, others specific.
Perhaps the answers to a few are buried in the literature, although we did not find them.

Lemma \ref{Lemma:VectorBounds} prompts us to consider symmetric tensor products of more than two vectors.  
If $\vec{x}_1,\vec{x}_2,\ldots, \vec{x}_n \in \h$, then 
$\norm{\vec{x}_1 \odot \vec{x}_2 \odot \cdots \odot \vec{x}_n} 
= \norm{ \S_n ( \vec{x}_1 \otimes \vec{x}_2 \otimes \cdots \otimes \vec{x}_n)}
\leq \norm{ \vec{x}_1 \otimes \vec{x}_2 \otimes \cdots \otimes \vec{x}_n }
= \norm{ \vec{x}_1} \norm{ \vec{x}_2} \cdots \norm{ \vec{x}_n}$.  Equality
occurs if $\vec{x}_1 = \vec{x}_2 = \cdots = \vec{x}_n$.  Thus, only lower bounds on 
$\norm{\vec{x}_1 \odot \vec{x}_2 \odot \cdots \odot \vec{x}_n}$ are of interest.
Here is a partial answer.

\begin{Lemma}\label{Lemma:ThreeVectors}
$\frac{1}{\sqrt{6}} \| \vec{x}_1 \| \| \vec{x}_2 \| \| \vec{x}_3 \| \leq \| \vec{x}_1 \odot \vec{x}_2 \odot \vec{x}_3 \| \leq \| \vec{x}_1 \| \| \vec{x}_2 \| \| \vec{x}_3 \|$ for $\vec{x}_1, \vec{x}_2, \vec{x}_3 \in \h$.
These inequalities are sharp
\end{Lemma}

\begin{proof}
The upper bound is discussed above.  Without loss of generality suppose
$\vec{x}_1, \vec{x}_2, \vec{x}_3 $ have unit norm.  Then
\begin{align}
     \!\!\!\!\!\!\!\!\!\!36\| \vec{x}_1 \odot \vec{x}_2 \odot \vec{x}_3 \|^2 
    &=   \sum_{\tau,\pi \in \Sigma_3} \langle \vec{x}_{\tau(1)} \otimes \vec{x}_{\tau(2)} \otimes \vec{x}_{\tau(3)} , \vec{x}_{\pi(1)} \otimes \vec{x}_{\pi(2)} \otimes \vec{x}_{\pi(3)} \rangle  \\
    & = 6 + \underbrace{\sum_{\tau \neq \pi} \langle \vec{x}_{\tau(1)} \otimes \vec{x}_{\tau(2)} \otimes \vec{x}_{\tau(3)} , \vec{x}_{\pi(1)} \otimes \vec{x}_{\pi(2)} \otimes \vec{x}_{\pi(3)}}_{c} \rangle , \label{eq:MidEqThing2}
\end{align}
in which $c\in \R$ is of the form
\begin{align}
c &= 6 (| \langle \vec{x}_2, \vec{x}_3 \rangle |^2 +  | \langle \vec{x}_1, \vec{x}_2 \rangle |^2   + | \langle \vec{x}_1, \vec{x}_3 \rangle |^2 ) \\
&\qquad+ \underbrace{6 \langle \vec{x}_1, \vec{x}_2 \rangle \langle \vec{x}_2, \vec{x}_3 \rangle \langle \vec{x}_3, \vec{x}_1 \rangle + 6  \langle \vec{x}_1, \vec{x}_3 \rangle \langle \vec{x}_3, \vec{x}_2 \rangle \langle \vec{x}_2, \vec{x}_1 \rangle}_{d}. 
\end{align}
Muirhead's inequality \cite[Ch.~2, Sec.~18]{Hardy} shows that for  $x,y,z \in [0,1]$,
\begin{equation}\label{eq:Muirhead}
    x^2 + y^2 + z^2 \geq 2 xyz.
\end{equation}
Let $x = |\langle \vec{x}_1, \vec{x}_2\rangle|$, $y = |\langle \vec{x}_2, \vec{x}_3\rangle|$,
and $z = |\langle \vec{x}_3, \vec{x}_1\rangle| $ in \eqref{eq:Muirhead} and get (since $d\in \R$)
\begin{equation*}
6 \big(| \langle \vec{x}_2, \vec{x}_3\rangle |^2 +  | \langle \vec{x}_1, \vec{x}_2\rangle |^2   +  |\langle \vec{x}_1, \vec{x}_3\rangle |^2 \big) \geq 12 |\langle \vec{x}_1, \vec{x}_2\rangle| |\langle \vec{x}_2, \vec{x}_3 \rangle| |\langle \vec{x}_3, \vec{x}_1 \rangle| \geq - d .
\end{equation*}
Thus, $c \geq 0$ and we obtain the desired lower bound.  If $\vec{x}_1, \vec{x}_2, \vec{x}_3$ are pairwise orthogonal, then $c=0$ in \eqref{eq:MidEqThing2}, so the lower bound is sharp.
\end{proof}

\begin{Problem}
Is $\frac{1}{\sqrt{n!}} \norm{ \vec{x}_1} \norm{\vec{x}_2} \cdots \norm{ \vec{x}_n} \leq
\norm{ \vec{x}_1 \odot \vec{x}_2 \odot \cdots \odot \vec{x}_n}$ for $\vec{x}_1, \vec{x}_2, \ldots,\vec{x}_n \in \h$?
\end{Problem}

Lemma \ref{Lemma:ThreeVectors} leads to an analogue of Theorem \ref{Theorem:Norm}.a for three operators.

\begin{Theorem}\label{Theorem:Norm2}
$ \displaystyle \frac{1}{\sqrt{6}}\sup_{\substack{\vec{x} \in \h \\ \|\vec{x}\| = 1}} 
\big\{ { \|A\vec{x}\| \|B\vec{x}\| \| C\vec{x} \|} \big\} \leq \| A \odot B \odot C \|$ for $A, B, C \in \B( \h )$.
\end{Theorem}

\begin{Problem}
For $A_1, A_2, \ldots, A_n \in \B(\h)$ is
\begin{equation}\label{QQ}
    \frac{1}{\sqrt{n!}}\sup_{ \substack{\vec{x} \in \h \\ \|\vec{x}\| = 1}} 
    \big\{ { \|A_1\vec{x}\| \|A_2\vec{x}\|\cdots \| A_n \vec{x} \|} \big\} \leq\| A_1 \odot A_2 \odot \cdots \odot A_n \| ?
\end{equation}
\end{Problem}

Proposition \ref{Proposition:Diagonals} provides the sharp inequalities 
$ \|L\| \| M \| (\sqrt{2}-1) \leqslant \| L \odot M \| \leq \|L\| \|M\|$ for diagonal operators $L,M$
(with respect to the same orthonormal basis).
Since the upper bound easily generalizes, the lower bound is of greater interest.

\begin{Problem}
Let $A_1, A_2, \ldots, A_n \in \B(\h)$ be diagonal operators (with respect to the same orthonormal basis).
Find a sharp lower bound, in the spirit of Proposition \ref{Proposition:Diagonals}, 
on $\norm{ A_1 \odot A_2 \odot \cdots \odot A_n}$ in terms of $\norm{A_1}, \norm{A_2},\ldots,\norm{A_n}$.
\end{Problem}

The Weyl--von Neumann--Berg theorem asserts that every normal operator on a separable Hilbert space
is the sum of a diagonal operator and a compact operator of arbitrarily small norm \cite[Cor.~II.4.2]{MR1402012}.
This suggests possible extensions to normal operators.

\begin{Problem}
Let $A_1, A_2, \ldots, A_n \in \B(\h)$ be commuting normal operators.
Find a sharp lower bound, in the spirit of Proposition \ref{Proposition:Diagonals}, 
on $\norm{ A_1 \odot A_2 \odot \cdots \odot A_n}$ in terms of $\norm{A_1}, \norm{A_2},\ldots,\norm{A_n}$.
\end{Problem}

Proposition \ref{Proposition:MultiDiagSpec} suggests the following.

\begin{Problem}
Let $A_1, A_2, \ldots, A_n \in \B(\h)$ be commuting normal operators.
Describe $\sigma(A_1 \odot A_2 \odot \cdots \odot A_n)$ (and its parts).
\end{Problem}

Let us now consider the unilateral shift $S$ and its adjoint.
Theorem \ref{Theorem:ShiftSpectra} identified the norm and spectrum of $S \odot S^*$
and $S \wedge S^*$.  What can be said about other combinations?

\begin{Problem}
Identify the norm and spectrum of arbitrary symmetric or antisymmetric
tensor products of $S$ and $S^*$ (for example, consider $S^2 \odot S \odot S^{*3}$ and $S^2 \wedge S \wedge S^{*3}$).
\end{Problem}

\begin{Problem}
Describe the norm and spectrum of $S_{\alpha} \odot S_{\alpha}^*$ and $S_{\alpha} \wedge S_{\alpha}^*$, 
in which $S_{\alpha}$ is a weighted shift operator.  What can be said if more factors are included?
\end{Problem}

Theorems \ref{Theorem:ShiftDiagSpec} and \ref{Theorem:DiagBackShift} answer some questions
about $S \odot M$ and $S^* \odot M$, in which $M=\diag(\mu_0,\mu_1,\ldots)$ is a diagonal operator.
However, a complete picture eludes us.

\begin{Problem}
Identify the norm and spectrum (and its parts) for $S \odot M$ and $S^* \odot M$.
\end{Problem}

The general problem suggested by the previous questions is the following.

\begin{Problem}
For $A_1,A_2,\ldots,A_n \in B(\h)$, describe the norm and spectrum (and its parts)  of
$A_1 \odot A_2 \odot \cdots \odot A_n$ and $A_1 \wedge A_2 \wedge \cdots \wedge A_n$.
\end{Problem}

There are countless other questions that can be raised.
For example, what can be said about symmetric tensor products of composition operators?

\bibliographystyle{plain}
\bibliography{SymmetricTensorOperator}
\end{document}